\documentclass[11pt]{article}
\usepackage{verbatim} 
\usepackage{graphicx}
\usepackage[utf8]{inputenc}
\usepackage{amsmath,amsfonts,amssymb,amscd,amsthm,txfonts,hyperref}
\usepackage{xcolor}
\newtheorem{theorem}{Theorem}[section]
\newtheorem{proposition}[theorem]{Proposition}
\newtheorem{lemma}[theorem]{Lemma}
\newtheorem{corollary}[theorem]{Corollary}
\newtheorem{remark}{Remark}[section]
\newtheorem{theo}{Theorem}

\theoremstyle{definition}
\newtheorem{definition}[theorem]{Definition}

\newcommand{\R}{\mathbb{R}}
\newcommand{\C}{\mathbb{C}}
\newcommand{\N}{\mathbb{N}}
\newcommand{\Z}{\mathbb{Z}}
\newcommand{\T}{\mathbb{T}}

\newcommand{\lap}{\bigtriangleup}

\newcommand{\E}{\mathbb E}

\newcommand{\mE}{\mathcal E}

\newcommand{\an}[1]{\langle #1 \rangle}
\newcommand{\mX}{\mathcal X}
\newcommand{\mXe}{\mathcal X^\varepsilon}

\begin{document}

\title{On Gibbs measure and weak flow for the cubic NLS with non-localised initial data}
\author{Federico Cacciafesta\footnote{SAPIENZA - Universit\'a di Roma, Dipartimento di Matematica, Piazzale A. Moro 2, I-00185 Roma, Italy - email:cacciafe@mat.uniroma1.it}\; and Anne-Sophie de Suzzoni\footnote{Universit\'e Paris 13, Sorbonne Paris Cit\'e, LAGA, CNRS ( UMR 7539), 99, avenue Jean-Baptiste Cl\'ement, F-93430 Villetaneuse, France - email: adesuzzo@math.univ-paris13.fr}}

\maketitle

\begin{abstract} 
In this paper we prove the existence of an invariant measure for the cubic NLS $$i\partial_t u + \lap u -  |u|^2 u = 0$$ on the real line in the sense that we prove the existence of a measure $\rho$ supported by non-localised functions such that there exists random variables $X(t)$ whose laws are $\rho$ (thus independent of $t$) and such that $t\mapsto X(t)$ is a solution to the cubic NLS. Our strategy for the proof is inspired by \cite{burqtzv} and relies on the application of Prokhorov and Skorokhod Theorems to a sequence of measures which are invariant under some approximating flows, as we proved in our previous \cite{lastbaby}. However, the work by Bourgain, \cite{B00} provides a stronger result than this one, as it gives almost sure strong solutions for the cubic NLS and the invariance of the measure can be deduced from it.
\end{abstract}

\tableofcontents

\section{Introduction}

The problem of building invariant measures under nonlinear flows for PDEs dates back to the pioneering work of Lebowitz-Rose-Speer \cite{LebRosSpe}, and was subsequently addressed by Bourgain in his seminal papers on KdV and Schr\"odinger on the torus \cite{B94}. 

In the subsequent years, a lot of effort has been spent in order to investigate the connections between invariant measures, and more in general the effect of "randomization" in PDEs, with various properties of the corresponding flows. This kind of probabilistic approach has been indeed successfully developed in several contests to significantly improve the existing deterministic theory: among the most remarkable results, we mention the papers \cite{burq1}-\cite{burq2} in which these ideas are developed to prove some supercritical well-posedness for the cubic wave equation. 

An attempt of providing a complete or even satisfying literature on the topic is out of our scope here, and we refer to the recent \cite{staffilani} for a fairly 
complete picture. We should anyway stress the fact that most of the available results concern evolution equations on compact  manifolds. The reason for this is somehow technical, as having a countable basis of eigenfunctions for the Laplacian turns to be a major advantage (suggesting a "natural" randomization) in the construction of an invariant measure, which can be roughly summarized by the following scheme: frequency truncation-Liouville theorem-uniform probability estimates. 
Randomization (and related questions) in a non-compact setting turns in fact to be much more complicated, and is a topic currently attracting a lot of attention from the community, which has produced some significant results in the very last years. We mention \cite{burthotzve} in which the authors consider a NLS on the real line with a well chosen external potential "trapping" the solution (see also \cite{benoh} in which modulation spaces are used and \cite{deng} for the 2D case), \cite{luhmen} in which supercritical well posedness 
for NLW on $\R^3$ is discussed (see also \cite{xu} for 3D NLW).

Our recent paper \cite{lastbaby} fits within this contest: there, we built a Gibbs measure for the cubic-type NLS	
\begin{equation}\label{NLS1}
i\partial_t u - \lap u + \chi |u|^2 u = 0,\qquad u:\R_t\times\R_x\rightarrow\C
\end{equation}
with a smooth interaction potential $\chi$ satisfying some smallness assumptions (namely, $0\leq\chi\lesssim \langle x\rangle^{-\alpha}$ and $|(1-\Delta)^{s_0/2}\chi|\lesssim\langle x\rangle^{-\alpha}$ for some $\alpha>1$, $s_0>1/4$), and proved it to be invariant under the flow of the equation above on a suitable topological $\sigma$-algebra. The main novelty of \cite{lastbaby} is represented by the randomization, as there is no trapping potential coming into play. Inspired by \cite{mckeanvan} (see also \cite{ansokg} for the Klein-Gordon equation), our randomization is therefore given by
$$
\phi(x)\cong\int_\R \frac{e^{inx}}{\sqrt{1+n^2}}dW_n(\omega)
$$
where $\omega$ is the random event and $W_n$ is a Brownian motion, which makes $\phi(x)$ It\^o  integral.
The presence of the function $\chi$ in \eqref{NLS1} is essentially technical, and heavily used in the convergence argument (the strategy to build the invariant measure consists in approximating the flow of \eqref{NLS1} with "approaching" equations on finite dimensional spaces, define invariant measures for them and then pass to the limit; this  requires several tools from local and global deterministic analysis).

The purpose of this paper is essentially to show that the function $\chi$ in \eqref{NLS1} can be removed, and thus to build a random variable which is a solution of the cubic NLS
\begin{equation}\label{NLS2}
i\partial_t u - \lap u +  |u|^2 u = 0,\qquad u:\R_t\times\R_x\rightarrow\C
\end{equation}
whose law law does not depend on time.

Our strategy is inspired by \cite{burqtzv}, in which the authors develop a clever machinery relying on two important results of measure theory, Prokhorov and Skorokhod Theorems, to build an invariant measure for some different dispersive flows on compact manifolds and obtaining, as an application, existence of solutions for the corresponding equations in certain spaces. This strategy comes from fluid mechanics, see for instance \cite{machin1,machin2} and reference therein. In \cite{burqtzv}, the authors adapt it to dispersive equations. We follow their proof.

We briefly summarize the main steps needed (all the details and the definitions will be made clear through the paper):

1) Given a PDE and its associated flow $\Phi(t)$, introduce an approximating problem such that its corresponding flow $\Phi_N(t)$ is global and such that we are able to construct an invariant measure $\rho_N$ on some topological space $X$, for which we have good control.

2) Define, for every $N$, the measure $\nu_N=\rho_N\circ \Phi_N^{-1}$ and show that it is \emph{tight} in some space $C([-T,T];Y)$ with $Y\supset X$. Therefore, the application of Prokhorov Theorem gives the existence of a measure $\nu$ on  $C([-T,T];Y)$ which is the weak limit of the sequence $\nu_N$.

3) Apply Skorokhod Theorem to construct a sequence of random variables converging to a solution of the initial problem.

We will thus apply this strategy to our setting with the aim of removing the interaction potential $\chi$ in \eqref{NLS1}. Equation \eqref{NLS1} will of course play the role of the "approximating problem", and we will use the invariant measure built in \cite{lastbaby} (although slightly changing the topological setting) as the approaching $\rho_N$. Anyway, we remark that the main difference with \cite{burqtzv} is in that we here work in a non compact setting: this will make the limiting argument significantly more complicated, and we will have to rely on some Feynman-Kac type results to make things work. 

We stress that the main difference with \cite{burqtzv} comes from the proof of the tightness of the family $(\nu_N)_N$. The tightness is deduced from uniform bounds on $\rho_N$ and the invariance of $\rho_N$ under the flow $\Phi_N$. This does not change for us. Nevertheless, the uniform bounds on $\rho_N$ in \cite{burqtzv} is based on the fact that $\rho_N$ can be written $d\rho_N(u) = F_N(u) dq(u)$ where $q$ is a well-known measure, often a Brownian bridge, and $F_N$ converges in $L^1(dq)$. This is not our case. The sequence of measure $\rho_N$ converges for path integral reasons towards a measure $\rho$ mutually singular with $q$,\cite{SIMfunint,GJ87} One of the main problem is here to understand this convergence in order to get uniform bounds on the sequence $(\rho_N)_N$. This requires to use Feynman-Kac or integral paths results. So, the novelty of this work consists in putting together Feynman-Kac type results \cite{SIMfunint,GJ87} with the Prokhorov-Skorokhod machinery.

We are now ready to state our main result (we postpone to the next subsection the definition of the functional spaces).
\begin{theo}\label{teo}
There exist a probability space $(\Omega,\mathcal{F}, P)$, a random variable $X=X(t)$ with values in $\mathcal{X}^\varepsilon_{\varphi, T}$ and a measure $\rho$ such that 
\begin{itemize}
\item
 For every $t$, the law of $X(t)$ is $\rho$ (thus, in particular, it does not depend on $t$);
\item
The random variable $X$ is almost surely a weak solution to \eqref{NLS2};
\item The measure $\rho$ is supported by non localised functions (not $L^2(\R)$).
\end{itemize}
\end{theo}

\begin{remark} With this strategy of proof, we cannot state a stronger result such as there exists a flow $\psi(t)$ of \eqref{NLS2} and a measure $\rho$ such that $\rho$ is invariant under $\psi(t)$ because the random variable $X(t)$ allows us only to define a weak flow of \eqref{NLS2} which in particular is not necessarily unique. We discuss the link between uniqueness and invariance in the Subsection \ref{finalinv}. Nevertheless, the measure $\rho$ is formally written $e^{-\mE(u)}du$ where $\mE(u)$ is the energy plus the mass, which makes it close to a Gibbs measure. Hence, at least in the sense of the above theorem, we get invariance of the "Gibbs" measure.

However, in \cite{B00}, Bourgain showed the uniqueness of the solution. Bourgain did not study the limiting measure, but its existence can be deduced from quantum field theory, such as in \cite{SIMfunint} as explained later or \cite{GJ87}, Chapter 3 and more precisely Theorem 3.3.3. Besides, combining the existence of the measure and the uniform convergence theorem of Bourgain, one should be able to prove the strong invariance of the limiting measure under the flow hence defined, giving a much stronger result. \end{remark}

\begin{remark} Considering only the kinetic part $dq$ of the approximating measures 
$$
d\rho_N(u) = F_N(u)dq(u),
$$
one can see that $q$ is somehow a large box limit. Indeed, $q$ is the limit of a sequence $q_L$, where $q_L$ is the law of 
$$
\varphi_L(x) = \sum_{k\in\Z} L^{-1/2} e^{ikx/L} (1+(k/L)^ 2)^{-1/2}g_k
$$
with $(g_k)$ independent centred normalised Gaussian variables. This limiting process has been explained in \cite{lastbaby} and consists in building a Wiener integral. The random variable $\varphi_L$ is built as a map of $2\pi L\T$ with $\varphi_L = (1-\lap)^{-1/2}\underline \varphi_L$ and 
$$
\underline \varphi_L (x) = \sum_{k\in \Z} \frac{e^{ikx/L}}{\sqrt L} g_k.
$$
We note that $(x\mapsto \frac{e^{ikx/L}}{\sqrt{2\pi L}})_k$ is an orthonormal basis of $L^2(2\pi L \T)$. The measure $dq$ is obtained by letting the size of the box (or torus) $2\pi L$ go to $\infty$. Besides, taking the mean value of the $L^2$ norm of $\varphi_L$ to the square gives something of order $L$ and hence diverges. This is a way to understand the non locality of the initial data.

However and as we have mentioned earlier, our final measure $\rho$ is mutually singular with $q$. Nevertheless, thanks to Feynman-Kac type results, we know that $\rho$ is invariant under translations and that when $u$ has the law $\rho$, the law of $u(x)$ is (not depending on $x$ and) absolutely continuous with regard to the Lebesgue measure. This is sufficient to prove that $\rho$ is supported by non localised functions, as we see in Propositions \ref{prop-nonloc}, \ref{prop-nonloc2}. In particular, $\rho$ is not supported by functions which are not in $L^2(\R)$.
\end{remark}

\begin{remark}Our proof is adaptable to other non-linearities. In fact, it is adaptable to other Hamiltonians. The sufficient conditions are given by the Feynman-Kac theory. But at least, one could consider a quintic non linearity. Or an equation with the same Hamiltonian but with a different dispersion relation. The fact that the measure is supported by non-localised functions may be explained by the fact that localised data may generate scattering. Then, the solutions would converge towards $0$ for some norm, but, as we see in Subsection \ref{invdiff}, invariance in a weaker norm often implies invariance in a stronger norm. This would contradict the invariance of the measure (it cannot both be invariant and converge to a Dirac delta centred in $0$ when time goes to $\pm\infty$). 
\end{remark}

Let us give some details on the plan of the paper. In the next section, we will provide the necessary notations, introducing the functional spaces and the measures we will deal with. In section \ref{prev} we will review and discuss some known results that will be the main ingredients in the proof of Theorem \ref{teo}: in particular, we will recall some generalities on Feynman-Kac theory for oscillatory processes and Prokhorov and Skorokhod Theorems. 
In subsection \ref{invdiff} we will show how to adapt our previous result of \cite{lastbaby} to the present functional setting. 
Section \ref{proofs} will be devoted to the proof of our main Theorem, that will be divided in several steps. First of all, we shall state two technical results (Lemmas \ref{lem-magicwithoutder}-\ref{lem-magicwithder}) in which we prove some uniform $N$ bounds for two crucial probability integrals; in subsection \eqref{conv} we prove the convergence of the invariant measures of \eqref{NLS1} for $N\rightarrow+\infty$ towards a limit 
$\rho$. 
Then, we prove the tightness of the family of measures $\nu_N$  (subsection \ref{tightsec}), the existence of a weak flow for equation \eqref{NLS2} as an application of Skorokhod Theorem (subsection \ref{weakflow}) and, eventually, we discussed the so-called invariance of the limit measure $\rho$ under the weak flow in subsection \ref{finalinv}.

\section{Notations}

In this section we fix up useful notations for the rest of the paper.

\subsection{Spaces}

Let $-2 \leq \sigma < -\frac74$. Let $\varepsilon > 0$ and $\alpha \in ]0,1[$. Let $T\in \R$.

Given any variable $x$, we use the standard notations for $\an{x} = \sqrt{1+x^2}$ and $D_x = \sqrt{1-\partial_x^2}$. We will denote with $S(\tau) = e^{i\tau \lap}$.

For $\varphi$ a non-negative increasing function, let $\mX_\varphi$ be the space induced by the norm
\begin{equation}\label{spaceX1} 
\|f\|_{\mX_\varphi} = \|(1+\varphi)^{-1}\an{x}^{-6} D_x^{\sigma}  f\|_{L^2}.
\end{equation}
Even though this is not one of the spaces that we used in \cite{lastbaby}  to prove the invariance of some measure $\rho$ under the flow of Schr\"odinger with a localised non linearity, one can prove that we have invariance in the topological $\sigma$ algebra of this space for density reasons. We take the regularity to be less than two orders where one has invariance and the weights to be three times what they should be such that the derivative in time of the solution to $i\partial_t u = -\lap u + |u|^2u$ is in this space too. In view of what has been done in \cite{burqtzv}, this loss of derivative is maybe superfluous.  The weight $\varphi$ is needed as an artefact of the proof and might be unnecessary.

For convenience reasons, we introduce the space $\mathcal Z_\varphi$ induced by the norm
\begin{equation}\label{spaceZ1} 
\|f\|_{\mathcal Z_\varphi} = \|\an{x}^{-2}(1+\varphi)^{-1} D_x^{\sigma+2}  f\|_{L^2} + \|\an{x}^{-2}(1+\varphi)^{-1/3} f\|_{L^6}.
\end{equation}

Let $\mXe_\varphi$ be the space induced by the norm
\begin{equation}\label{spaceY1}
\|f\|_{\mXe_\varphi} = \|(1+\varphi)^{-1}\an{x}^{-6(1+\varepsilon)} D_x^{\sigma(1+\varepsilon)} f\|_{L^2}.
\end{equation}
We will prove later that the balls of $\mX_\varphi$ are compact in $\mXe_\varphi$.

Let $\mX_{T,\varphi}$ and $\mXe_{T,\varphi}$ be the spaces defined as
\begin{equation}\label{spaceXY}
\mX_{T,\varphi} = \mathcal C^\alpha([-T,T],\mX_\varphi)\mbox{ and } \mXe_{T,\varphi} = \mathcal C([-T,T],\mXe_\varphi)
\end{equation}
where the index $\alpha$ is related to Lipschitz continuity in the sense that
$$
\|f\|_{\mX_{T,\varphi}} = \sup_{t_1,t_2 \in [-T,T]} \frac{\|f(t_1) - f(t_2)\|_{\mX_\varphi}}{|t_1-t_2|^\alpha} + \|f\|_{L^\infty([-T,T],\mX_\varphi)}.
$$
The idea is that the balls of $\mX_{T,\varphi}$ are compact in $\mXe_{T,\varphi}$.

Let $\textbf{m}$ be a measure and $p \in [1,\infty]$. By $L_{\textbf{m}}^p$ we denote the space induced by the norm
$$
\|F\|_{L^p_{\textbf{m}}} = \Big( \int |F(u)|^p d\textbf{m}(u) \Big)^{1/p}
$$
if $p<\infty$ or 
$$
\|F\|_{L^\infty_{\textbf{m}}} = \sup \{\lambda \geq 0\, |\, \textbf{m} (|F|\geq \lambda) \neq 0\}.
$$

\subsection{Measures}

Let $\mu_N$ be the measure defined as 
$$
d\mu_N (u) = D_N^{-1} e^{-\frac12 \int_{-N}^N |u(x)|^4dx} dq(u)
$$
where $q$ is the complex valued oscillator process given in the book by Simon \cite{SIMfunint}. We give more details about this process in Subsection \ref{subsec-simon}, and $D_N$ is the $L^1_q$ norm of $e^{-\frac12 \int_{-N}^N |u(x)|^4dx}$. We remark that $D_N$ goes to $0$ when $N$ goes to $\infty$.  

Let $\chi_N$ be $\mathcal C^\infty$ functions with compact supports such that for all $x\in \R$, $\chi_N(x) \in [0,1]$,
\begin{equation}\label{defchiN} 
\chi_N(x) = 1 \mbox{ on } [-N,N] \mbox{ and } \chi_N(x) = 0 \mbox{ outside } [-N-D_N^3, N+D_N^3].
\end{equation}

We call $\rho_N$ the invariant measure defined in \cite{lastbaby} under the flow of
\begin{equation}\label{NLSN}
i\partial_t u = -\lap u + \chi_N |u|^2u.
\end{equation}
We call $\psi_N$ the flow of this equation.

Let $\nu_N$ be the measure defined on the topological $\sigma$ algebra of $\mX_{T,\varphi}$ as for all $A$
\begin{equation}\label{defnuN}
\nu_N(A) = \rho_N\Big( \big\{ u_0 \, \Big| \, t\mapsto\psi_N(t) u_0 \in A\big\}\Big).
\end{equation}

\section{Previous results and corollary}\label{prev}

\subsection{Convergence in the whole line}\label{subsec-simon}

We begin with the following definition.
\begin{definition}\label{oscdef}
A family $\{q(x)\}_{x\in\mathbb{R}}$ of Gaussian random variables is called an \emph{oscillator process} or Ornstein-Uhlenbeck process if
$$
\mathbb{E}(q(x)q(y))=\frac12 e^{-|x-y|}.
$$
We will denote with $dq$ the measure on paths $\omega(x)$ associated to the oscillator process.
\end{definition}

In analogy with what happens with Brownian motions, it is natural to link oscillator processes with some suitable semi group. We explain this connection in the following result.

\begin{proposition}\label{semig}
Let $L_0=-\frac12\frac{d^2}{dx^2}+\frac12x^2-\frac12$ and $\Omega_0(x)=\pi^{-1/4}e^{-(1/2)x^2}$ so that $L_0\Omega_0=0$ and $\int|\Omega_0|^2=1$. Let moreover $f_0,\dots f_n\in L^\infty (\mathbb{R})$ and let $-\infty<y_0<\dots y_n<\infty$. Then
\begin{equation*}
\mathbb{E}(f_0(q(y_0)),\dots f_n(q(y_n)))=(\Omega_0,M_{f_0} e^{-x_1L_0}M_{f_1}\dots e^{-x_nL_0}M_{f_n}\Omega_0)_{L^2}
\end{equation*}
where $x_i=y_i-y_{i-1}>0$, $(\cdot,\cdot)_{L^2}$ denotes here the standard $L^2$ scalar product and $M_f$ the multiplication operator $M_f g(x)=f(x)g(x)$.
\end{proposition}

\begin{proof}
See \cite{SIMfunint} Theorem 4.7 pag. 37.
\end{proof}

\begin{remark}
Proposition \ref{semig} yields an explicit kernel $Q_x(u_1,u_2)$ for the semi group $e^{-xL_0}$. Following the lines of the proof indeed we have
$$
e^{-xL_0}f(u_1)=\int Q_x(u_1,u_2)f(u_2)du_2
$$
where
$$
Q_x(u_1,u_2)=\frac1{\sqrt{\pi(1-e^{-2x})}}\exp\left(-\frac{\frac12(u_1^2+u_2^2)(1+e^{-2x})-2e^{-x}u_1u_2}{1-e^{-2x}}\right)
$$
This is known as Mehler's formula.

We remark that $Q_x$ is smooth and that for all $t\geq 0$ and $u_1\neq u_2$
\begin{equation}\label{estimQ1}
Q_x(u_1,u_2) \lesssim 1+ |u_1-u_2|^{-1} \mbox{ and } |\partial_x Q_x(u_1,u_2)| \lesssim 1+ |u_1-u_2|^{-3/2}|u_1+u_2|.
\end{equation}
We also have that $\partial_xQ(x=0) = 0$ and 
\begin{equation}\label{estimQ2}
 |\partial_x^2 Q_x(u_1,u_2)| \lesssim 1+ |u_1-u_2|^{-5/2}|u_1+u_2|^2.
\end{equation}
And finally, we get that $\partial_x^2 Q_x(x=0) = 0$ and 
\begin{equation}\label{estimQ3}
 |\partial_x^3 Q_x(u_1,u_2)| \lesssim 1+ |u_1-u_2|^{-7/2}|u_1+u_2|^3.
\end{equation}
\end{remark}

\begin{remark}
Actually, minor modifications in the proof allow to adapt this result to higher dimensions: in this case the natural semi group will be indeed given by $L_0=-\frac12\Delta+\frac12x^2-\frac12$.
\end{remark}

The next step consists now in giving the analogue of Proposition \ref{semig} in a slightly more general setting, i.e. to relate the semi group $e^{-xL}$ with $L=L_0+V$ for some suitable potential $V$ to path integrals. Results of this kind have been widely investigated in literature, especially in the case of Brownian motion, and are usually referred to as \emph{Feynman-Kac formulas}. In what follows $V$ will be any polynomial bounded from below, so that $E(V)=\inf {\rm spec}(L_0+V)$ is a simple eigenvalue with an associated strictly positive eigenvector $\Omega_V$ (some more general potentials can be considered, but we do not strive to cover the most general case here as discussed in \cite{ReedSimonIV}). We will denote with $\hat{L}=L_0+V-E(V)$.

\begin{definition}\label{Pphi}
We define the $P(\phi)_1$-process as the stochastic process with joint distribution of $q(x_1),\dots q(x_n)$ ($x_1<\dots <x_n$):
$$
\Omega_V(u_1)\Omega_V(u_n)e^{-y_1\hat {L}}(u_1,u_2)\dots e^{-y_{n-1}\hat{L}}(u_{n-1},u_n)
$$ 
where $e^{-y\hat{L}}(a,b)$ is the integral kernel of $e^{-s\hat{L}}$ and $y_i=x_{i+1}-x_i$. We will denote with $d\rho_V$ the corresponding measure.
\end{definition}

\begin{remark}\label{decayOmega}
In view of what will follow in the next section, it is important to give some estimate on the ground state $\Omega_V(x)$ (which is a regular function), with $V\cong |x|^4$. This can be done by means of the so called WKB approximation scheme, which gives the asymptotic behaviour $\Omega_{|x|^4}(x)\cong e^{-|x|^3}$ for $|x|\rightarrow+\infty$. We refer to \cite{benderwu}, \cite{turbiner05} for details.
\end{remark}

Therefore, we are ready to state the following result.

\begin{theorem}\label{feykac}[Feynman-Kac]
For any smooth and bounded test function $G : \mathcal C(\R,\C) \rightarrow \R$, we have 
\begin{equation*}
\int G(u) d\rho_V(u)=\lim_{N\rightarrow+\infty}D_N^{-1}\int G(q) e^{-\int_{-N}^NV(q(s))ds}dq
\end{equation*}
where $D_N$ is the $L^1(dq)$ norm of $e^{-\int_{-N}^NV(q(s))ds}$
\end{theorem}

\begin{proof}
See \cite{SIMfunint} Theorems 6.7 and 6.9 pag 58. Notice that, by mimicking the proof of Theorem 6.1 there, it is possible to deal also with the complex case.
\end{proof}

\begin{remark}
This Theorem implies in particular the convergence of the sequence $\mu_N$ as defined in the introduction: the choice of the potential $V=|x|^4$ falls indeed within the assumptions we made for Definition \ref{Pphi} and therefore for applying Feynman-Kac Theorem. In what follows, we will omit the dependence on $V=|x|^4$ for the limit measure simply denoting it with $\rho$.
\end{remark}

The reason for introducing all this framework is in the following result, in which we show that the Gaussian part of the invariant measure for NLS built in \cite{lastbaby} is a complex valued oscillator process in the sense of the next proposition.

\begin{proposition}\label{qvar}
Suppose $W_n(\omega)$ is the reunion of two complex, independent Wiener processes in $n$, $W^1_n(\omega)$, $W^2_n(\omega)$ and let
\begin{equation}\label{ranvar}
\phi(x)=\frac1{2\pi}\int_\R \frac{e^{inx}}{\sqrt{1+n^2}}dW_n(\omega)
\end{equation}
where $x\in\R$ and $\omega\in\Omega$. Then it is possible to decompose 
$$\phi(x)=\phi_1(x)+i\phi_2(x)$$
where $\phi_1$, $\phi_2$ are real-valued and independent. Moreover, each $\phi_j(x)$ is an oscillator process, as in Definition \ref{oscdef}. 

\end{proposition}

\begin{remark} Before we prove this proposition, we remark that the process $W_n$ is a random Gaussian field, see \cite{simonPphi2}, such that $W_0 = 0$ and 
$$
\E(\overline{dW_{n_1}} dW_{n_2}) = dn_1 \delta(n_1-n_2)
$$
or equivalently
$$
\E(\overline{W_{n_1}}W_{n_2})= \left \lbrace{\begin{array}{ll}
0 & \mbox{ if } n_1n_2 <0\\
\min (|n_1|,|n_2|) & \mbox{ otherwise.}
\end{array}} \right.
$$
\end{remark}

\begin{proof}
In view of our assumption on $W_n(\omega)$, we have
$$
\phi(x)=\int_{\R}\frac{\cos(n x) dW^1_n(\omega)-\sin(n x)dW^2_n(n)}{\sqrt{1+n^2}}
$$
$$
+i\int_{\R}\frac{\sin(n x) dW^1_n(\omega)+\cos(n x)dW^2_n(n)}{\sqrt{1+n^2}}
$$
$$
=\phi_1(x)+i\phi_2(x).
$$
To prove the independence, we rely on the Ito isometry to write
\begin{eqnarray*}
\mathbb{E}(\phi_1(x)\phi_2(y)))&=&\int_\R\frac{\cos(n x)\sin(n y)}{1+n^2}dn-
\int_\R\frac{\cos(n x)\sin(n y)}{1+n^2}dn
\\
&=&
\int_{\R}\frac{\sin (n(x-y))}{1+n^2}dn
\\
&=&0.
\end{eqnarray*}
Since $\phi_1$, $\phi_2$ are centred Gaussian variables, this implies that $\phi_1$ and $\phi_2$ are independent.

We now come to the second part of the proposition. First of all, we observe that
\begin{eqnarray}\label{exppro}
\mathbb{E}(\phi_1(x)\phi_1(y))&=&\mathbb{E}(\phi_2(x)\phi_2(y))
\\
\nonumber
&=&
\int_\R\frac{\cos(nx)\cos(ny)+\sin(nx)\sin(ny)}{1+n^2}dn
\\
\nonumber
&=&
\int_\R \frac{\cos(n(x-y))}{1+n^2}dn.
\end{eqnarray}
Let us then consider the function $\displaystyle F(x)=\int_\R \frac{\cos(nx)}{1+n^2}dn$: we aim to prove that
\begin{equation}\label{parti1}
F(x)=\alpha \sinh(|x|)+\beta \cosh (x).
\end{equation}
Let $\psi\in C^\infty_c$ be a test function. We have by considering Fourier transform
\begin{eqnarray*}
\langle F,(1-\partial_x^2)\psi\rangle&=&
\big\langle \int_\R \cos(nx),\psi\big\rangle
\\
&=& {\rm Re}\left[\int_\R dx \psi(x)\int_\R e^{inx}dn\right]
\\
&=&
{\rm Re}\left[\int_\R dn\int_\R e^{inx}\psi(x)dx\right]
\\
&=&
\sqrt{2\pi}{\rm Re}\left[\int_\R\hat{\psi}(n)dn\right]
\\
&=&
2\pi \psi(0).
\end{eqnarray*}
On the other hand,
\begin{eqnarray*}
\langle \sinh(|x|),(1-\partial_x^2)\psi\rangle&=&
\int_0^{+\infty}\sinh (x)(1-\partial_x^2)\psi-
\int_{-\infty}^0\sinh (x)(1-\partial_x^2)\psi.
\\
&=& 
I+II.
\end{eqnarray*}
The first integral gives
$$
I=\int_o^{+\infty}\sinh (x)\psi-\int_0^{+\infty}\sinh (x)\partial^2_x\psi=I_1-I_2
$$
where integrating by parts
\begin{eqnarray*}
I_2&=&\sinh (x)\partial_x\psi\big|^{+\infty}_0-\int_0^{+\infty}\cosh (x)\partial_x\psi
\\
&=&
-\cosh(x)\psi\big|^{+\infty}_0+\int_0^{+\infty}\sinh(x)\psi
\\
&=&
\psi(0)+I_1
\end{eqnarray*}
and thus $I=-\psi(0)$.

Analogously,
\begin{eqnarray*}
-II&=&\int^0_{-\infty}\sinh(x)(1-\partial^2_x)\psi
\\
&=&II_1-II_2,
\end{eqnarray*}
and
\begin{eqnarray*}
II_2&=& \int_{-\infty}^0 \sinh(x)\partial^2_x\psi
\\
&=&
\sinh(x)\partial_x\psi\big|^0_{-\infty}-\int_{-\infty}^0\cosh(x)\partial_x\psi
\\
&=&
-\cosh(x)\psi\big|^0_{-\infty}+\int_{-\infty}^0\sinh(x)\psi
\\
&=&
-\psi(0)+II_1
\end{eqnarray*}
which implies $-II=\psi(0)$. Therefore, we have showed that
$$
\langle \sinh(|x|),(1-\partial^2_x)\psi\rangle=-2\psi(0).
$$
On the other hand,
$$
\langle \cosh(x),(1-\partial_x^2)\psi\rangle=\langle (1-\partial_x^2)\cosh(x),\psi\rangle=0.
$$
Putting all together, we thus have
$$
F(x)=-\pi(\sinh(|x|)-\cosh(x)),
$$
as $F(0)=\pi$. Hence, for $x\geq0$, we have
$$
F(x)=\pi\left(\frac{e^x+e^{-x}-e^x+e^{-x}}2\right)=\pi e^{-x}.
$$
Getting back to \eqref{exppro} this gives, when $x\geq y$, 
\begin{equation*}
\mathbb{E}(\phi_j(x)\phi_j(y))=\pi e^{-(x-y)},\qquad j=1,2
\end{equation*}
and this concludes the proof.

\end{proof}

\begin{remark}We can also remark that $e^{-|x-y|}$ is the Green function of the operator $(1-\partial_x^2)^{-1}$.\end{remark}

As a concluding result for this subsection, we give the following Proposition which is just a consequence of what we have seen so far.

\begin{proposition}\label{ovvio}
Let $\rho_N$ be the invariant measure defined in \cite{lastbaby}. Then
\begin{equation}\label{invdq}
\displaystyle
d\rho_N(u)=\frac1{D'_N} e^{-\int \chi_N|u(x)|^4dx}dq(u).
\end{equation}
where $D'_N$ is the $L^1(dq)$ norm of $e^{-\int \chi_N|u(x)|^4dx}$.
\end{proposition}

\paragraph{Additional remarks on where $\rho$ is supported}

We wish to prove that $\rho$ is supported by functions which are not localised, in particular, in the sense that they are not $L^2$. For this, we remark that, thanks to the description of $\rho$ as the $P(\phi)_1$ process given by Definition \ref{Pphi}, we have that $\rho$ is invariant under translations, that is for all test functions $F$ and all $x_0 \in \R$ : 
$$
\int F(f)d\rho(f) = \int F(f_{x_0})d\rho(f)
$$
with $f_{x_0}(x) = f(x-x_0)$. Besides, $f$ is $\rho$-almost surely continuous and the law of $f(x)$ is absolutely continuous with regard to the Lebesgue measure and with density $\Omega_V^2$, that is
$$
\rho(f(x) \in [f,f+df]) = \Omega_V^2(f)df.
$$

\begin{proposition}\label{prop-nonloc} The measure $\rho$ is supported by non localised functions in the sense that $\rho$ almost surely $f(x)$ does not go to $0$ when $x$ goes to $\infty$.
\end{proposition}

\begin{proof}We compute the probability such that $f(x)$ goes to $0$ when $x$ goes to $\infty$. We have 
$$
\rho(f(x) \rightarrow 0) = \rho(\forall \varepsilon >0 \exists R \mbox{ such that } \forall x\geq R \; ,\; |f(x)|\leq \varepsilon).
$$
Writing everything in terms of sets, we have 
$$
\rho(f(x) \rightarrow 0) = \rho\Big(\bigcap_{\varepsilon >0} \bigcup_{ R\in \R}  \bigcap_{ x\geq R} ( |f(x)|\leq \varepsilon)\Big).
$$
Because of decreasing continuity of $\rho$, we have 
$$
\rho(f(x) \rightarrow 0) = \inf_\varepsilon\rho\Big(\bigcup_{ R\in \R}  \bigcap_{ x\geq R} ( |f(x)|\leq \varepsilon)\Big).
$$
Writing $\bigcup_{ R\in \R}  \bigcap_{ x\geq R}$ as a $\liminf$, we get
$$
\rho(f(x) \rightarrow 0) = \inf_\varepsilon\rho(\liminf_{x\rightarrow \infty}( |f(x)|\leq \varepsilon)).
$$
We use Fatou's lemma to get
$$
\rho(f(x) \rightarrow 0) \leq \inf_\varepsilon \liminf_{x\rightarrow \infty} \rho( |f(x)|\leq \varepsilon)
$$
and the invariance of $\rho$ under translations to get $\rho( |f(x)|\leq \varepsilon) = \rho( |f(0)|\leq \varepsilon)$ and thus
$$
\rho(f(x) \rightarrow 0) \leq \inf_\varepsilon  \rho( |f(0)|\leq \varepsilon).
$$
Finally, we use again the decreasing continuity of $\rho$ to get
$$
\rho(f(x) \rightarrow 0) \leq  \rho( f(0)=0).
$$
The law of $f(0)$ being absolutely continuous with regard to the Lebesgue measure, we have that 
$$
\rho(f(x) \rightarrow 0) =0.
$$
\end{proof}

\begin{proposition}\label{prop-nonloc2} The measure $\rho$ is supported by non localised functions in the sense that $\rho$ almost surely $f$ does not belong to $L^2$.
\end{proposition}

\begin{proof} Let $R\in \Z$ and let 
$$
\|f\|_R^2 = \int_{R}^{R+1} |f(x)|^2dx.
$$ 
If $f$ belongs to $L^2$ then the series of general term $\|f\|_{R}^2$ converges and hence $\|f\|_R$ goes to $0$ when $R$ goes to $\infty$. Therefore
$$
\rho( \|f\|_{L^2} <\infty) \leq \rho (\|f\|_R \rightarrow 0).
$$
For the same reasons as in the proof of Proposition \ref{prop-nonloc}, we have
$$
\rho( \|f\|_{L^2} <\infty) \leq \rho( \|f\|_0 = 0).
$$
Since $f$ is $\rho$ almost surely continuous, we get that $\|f\|_0=0$ almost surely implies $f(0) = 0$ and thus
$$
\rho( \|f\|_{L^2} <\infty) \leq \rho( f(0) = 0) = 0.
$$
\end{proof}

\subsection{Prokhorov's theorem}

In this section we present a classical result of probability theory, known as Prokhorov Theorem, that represents a crucial tool in our convergence argument, and essentially connects the concepts of weak compactness and tightness. We refer to \cite{skoro}, \cite{prok} for all the details and deeper insight on the topic. First of all, we recall the following
\begin{definition} [Weak compactness]

Let $S$ be a metric space. A family $(\mathbf{m}_N)_{N\geq1}$ of probability measures on the Borel $\sigma$-algebra $\mathcal{B}(S)$ is said to be \emph{weakly compact} if from any sequence $\mathbf{m}_N$, $N=1,2,\dots$ of measures from the family one can extract a weakly convergent subsequence $\mathbf{m}_{N_k}$, $k=1,2,\dots$ that is $\mathbf{m}_{N_k}\rightarrow \mathbf{m}$ for some probability measure $\mathbf{m}$.
\end{definition}

\begin{remark}
Note that the definition does not require $\mathbf{m}\in(\mathbf{m}_N)_N$.
\end{remark}

\begin{remark} We recall that weak convergence means that for all $F : S \rightarrow \R$ Lipschitz continuous and bounded we have 
$$
\E_{\mathbf m_{N_k}} (F)\rightarrow \E_{\mathbf m} (F).
$$
The convergence in law is stronger as it means that for all $F : S \rightarrow \R$ bounded we have 
$$
\E_{\mathbf m_{N_k}} (F)\rightarrow \E_{\mathbf m} (F).
$$
\end{remark}

\begin{definition} [Tightness]
Let $S$ be a metric space and $(\mathbf{m}_N)_{N\geq1}$ a family of probability measures on the Borel $\sigma$-algebra $\mathcal{B}(S)$. The family $(\mathbf{m}_N)_N$ is said to be \emph{tight} if for any $\varepsilon>0$ it is possible to find a compact set $K_\varepsilon\subset S$ such that for all $N>1$, $\mathbf{m}_N(K_\varepsilon)\geq 1-\varepsilon$.
\end{definition}

\begin{theorem}[Prokhorov Theorem]\label{prok}
If a family $(\mathbf{m}_N)_{N\geq1}$ of probability measures on a metric space $S$ is tight, then it is weakly compact. Moreover, on a separable complete metric space the two notions are equivalent.
\end{theorem}

\begin{proof}
See e.g. \cite{prok}, pag 114.
\end{proof}

Let us now explain how we will make use of this Theorem. We have already introduced the measure $\nu_N$ defined on the topological $\sigma$-algebra of $\mathcal{X}^\varepsilon_{T,\varphi}$ as the image measure by the map
\begin{eqnarray*}
\mathcal{X}_{\varphi}&\rightarrow& \mathcal{X}_{T,\varphi}
\\
v&\rightarrow& (t\mapsto \psi_N(t)(v));
\end{eqnarray*}
notice that in particular, for any measurable function $F:\mathcal{X}_{T,\varphi}\rightarrow\mathbb{R}$,
\begin{equation}\label{invarmeasf}
\int_{\mathcal{X}_{T,\varphi}}F(u)d\nu_N=\int_{\mathcal{X}_\varphi}F(\psi_N(t)(v))d\rho_N.
\end{equation}
The idea now is to show that the sequence of measures $\{\nu_N\}_N$ is tight in the space $\mathcal{X}^\varepsilon_{T,\varphi}$ for any $\varepsilon>0$ (this will be done in details in subsection \ref{subsec-tightness}). Therefore, the application of Theorem \ref{prok} yields the weak convergence (up to a subsequence) of $\{\nu_N\}_N$ towards a measure $\nu$ on $\mathcal{X}^\varepsilon_{T,\varphi}$.

\subsection{Skorokhod's theorem}\label{subsec-sko}

In this subsection, we give and comment Skorokhod's theorem and explain how we use it to get the existence of a weak solution to the cubic defocusing Schr\"odinger equation \eqref{NLS2}.


\begin{theorem}[Skorokhod] 
Let $S$ be a metric space and let $(\textbf{m}_N)_N$ be a sequence of measures on $S$ converging weakly towards a measure $\textbf{m}$ on $S$. We assume that the supports of $\textbf{m}_N$ and $\textbf{m}$ are separable. Then, there exists a probability space and a sequence of random variables $(X_N)_N$ and a random variable $X$ on this probability space such that \begin{itemize}
\item the law of $X_N$ is $\textbf{m}_N$,
\item the law of $X$ is $\textbf{m}$,
\item the sequence $(X_N)_N$ converges almost surely towards $X$.
\end{itemize}
\end{theorem}

We refer to \cite{skoro} for the proof and some applications.

Let us now give some remarks. First, we have that the space $\mXe_{\varphi,T}$, given by \eqref{spaceXY} is separable. Assuming that we have proven that the sequence of measures $\nu_N$ converges weakly, which we deduce in Subsection \ref{subsec-tightness} from Prokhorov's theorem, we get the existence of a sequence of random variables $X_N$ of law $\nu_N$ which converges towards $X$ of law $\nu$ the limit of $(\nu_N)_N$ up to a subsequence.

We now explain why $X_N$ can be written $X_N(t,x) = \psi_N(t)(Y_N)(x)$ such that $Y_N(x) = X_N(0,x)$ and its law is $\rho_N$.

\begin{proposition}\label{prop-linkrvmeas} Assume that $X_N$ is a random variable on a probability space $(\Omega, \mathcal F, P)$ with value in $\mXe_{\varphi,T}$ of law $\nu_N$. Let $Y_N = X_N(t=0)$. Then, $P$-almost surely we have $X_N(t) = \psi_N(t) Y_N$ and the law of $Y_N$ is $\rho_N$.
\end{proposition}

\begin{proof} Let $A$ be the set
$$
A=\{\omega \in \Omega \, |\, \forall t \in \R ,\, X_N(t)(\omega) = \psi_N(t) X_N(t=0)(\omega)\} .
$$
We can rewrite $A$ as
$$
A  =  \{\omega \in \Omega \, |\, \exists u_0\in \mXe_\varphi\, \forall t \in \R ,\, X_N(t)(\omega) = \psi_N(t) u_0\}.
$$
Indeed, if $\omega \in A$ then there exists $u_0 = X_N(t=0)(\omega)$ in $\mXe_\varphi$ such that $X_N(t)(\omega) = \psi_N(t) u_0$. Conversely, if there exists $u_0 \in \mXe_\varphi$ such that $X_N(t)(\omega) = \psi_N(t) u_0$ then $X_N(0)(\omega) = \psi_N(0)u_0 = u_0$.

Hence, as the law of $X_N$ is $\nu_N$,
$$
P(A) = \nu_N \Big( \{ u \,|\, \exists u_0 \in \mXe_\varphi \, u(t) = \psi_N(t) u_0\}\Big) = \nu_N \Big(\psi_N(t) (\mXe_\varphi)\Big).
$$
And we recall from the definition of $\nu_N$ that
$$
\nu_N(B ) = \rho_N ( u_0 \,|\, \psi_N(t) u_0 \in B).
$$
Therefore 
$$
P(A) = \rho_N ( u_0 \, |\, \psi_N(t) u_0 \in \psi_N(t) (\mXe_\varphi)) = \rho_N(\mXe_\varphi) = 1.
$$
In other terms, $P$ almost surely $X_N(t)(\omega) = \psi_N(t) X_N(t=0)(\omega)$.

Let us prove that the law of $Y_N = X_N(t=0)$ is $\rho_N$. Let $A$ be a measurable set of $\mXe_\varphi$. We have 
$$
P(Y_N \in A ) = P( X_N(t=0) \in A).
$$
Since the law of $X_N$ is $\nu_N$, we get
$$
P(Y_N \in A ) =\nu_N ( u | u(t=0) \in A).
$$
And given the definition of $\nu_N$ we have 
$$
P(Y_N \in A ) = \rho_N ( \{u_0 \, |\, \psi_N(t)u_0 \in \{u | u(t=0) \in A\}\} = \rho_N(A).
$$
Hence the law of $Y_N$ is $\rho_N$.\end{proof}

\begin{proposition}\label{prop-invarrv} Under the assumptions of Proposition \ref{prop-linkrvmeas}, for all $t\in \R$ the law of $X_N(t)$ is $\rho_N$.\end{proposition}

\begin{proof} As $X_N(t) = \psi_N(t) Y_N$, we have that the law of $X_N(t)$ is the image measure of $\rho_N$ under $\psi_N(t)$ but since $\rho_N$ is invariant under $\psi_N$, we get that the law of $X_N(t)$ is $\rho_N$.
\end{proof}

The idea is now that as $X_N (t) = \Psi_N(t) Y_N$, the random variable $X$ is a weak solution of the cubic non linear Schr\"odinger equation \eqref{NLS2}, on the support of the limit measure $\rho$, see Subsection \ref{subsec-weaksol}.

\subsection{Invariance of \texorpdfstring{$\rho_N$}{rhoN}}\label{invdiff}

In this subsection, we recall the result of \cite{lastbaby}, and explain the density argument which makes $\rho_N$ invariant under $\psi_N$ in $\mathcal Z_\varphi$.

In \cite{lastbaby} , we proved that the measures $\rho_N$ were invariant under the flow $\psi_N$ for some topology $\mathcal Y_s$ induced by the norm
$$
p_s (f) = \|\an{t}^{-2}\an{x}^{-2}D^s S(t) f\|_{L^2(t\in \R,x\in \R)} 
$$
for $s < -1/2$. Indeed, as $\chi_N$ is $\mathcal C^\infty$ with compact support, it satisfies the hypothesis of Subsection 1.1 in \cite{lastbaby} . This means that for all measurable bounded function $F$ of $\mathcal Y_s$ and all times $t\in \R$, we have 
$$
\E_{\rho_N}(F\circ\psi_N(t)) = \E_{\rho_N}(F).
$$
We recall that $S(t) = e^{-it\lap}$.

We wish to prove that this property is also true in $\mathcal Z_\varphi$. Namely, that for all measurable bounded function $F$ of $\mathcal Z_\varphi$,
$$
\E_{\rho_N}(F\circ\psi_N(t)) = \E_{\rho_N}(F).
$$

For this, we need the following lemmas.

\begin{lemma}\label{lem-welldef} For all non negative and increasing function $\varphi$, and for $\rho_N$ almost all $u$, we have $\psi_N(t) u \in \mathcal Z_\varphi$. \end{lemma}

\begin{proof} This is a consequence of Proposition 4.4, bullet 3 in \cite{lastbaby}. Indeed, with a control of $\psi_N(t_n) u$ at discrete well-chosen times, one can apply the contraction argument for the well-posedness and deduce that 
$$
\Psi_N(t) u = \psi_N(t)u - S(t)  u
$$
belongs to $H^{s_\infty}$ with $s_\infty$ given in Subsection 1.1 of \cite{lastbaby} for all $t\in [t_n,t_{n+1}]$ and ultimately at all time. Given that $s_\infty$ can be chosen as close as but strictly less than $1/2$ and that $H^{s_\infty}$ is embedded in $L^6$ if $s_\infty \geq 1/3$ and hence that $H^{s_\infty}$ is embedded in $\mathcal Z_\varphi$ if $s_\infty \geq \max( 1/3,2+\sigma)$, we get that $\Psi_N(t) u$ belongs $\rho_N$ almost surely to $\mathcal Z_\varphi$. 

The fact that $S(t) u$ belongs $\rho_N$ almost surely to $\mathcal Z_\varphi$ is a consequence of Proposition 2.5 in \cite{lastbaby}.
\end{proof}

\begin{lemma}\label{lem-meas} Let $w$ be a $\mathcal C^\infty$ function of $\R$ with compact support. There exists a constant depending on $w$, $C(w)$, such that for all $u\in \mathcal Y_s$, 
$$
\|w* u\|_{\mathcal Z_\varphi}\leq p_{s}(u) .
$$
\end{lemma}

\begin{proof} We write $\mathcal Z_\varphi =  \mathcal Z_\varphi^2 \cap \mathcal Z_\varphi^6$ where $\mathcal Z_\varphi^2$ is the space induced by the norm 
$$
\|\an{x}^{-2}(1+\varphi)^{-1} D_x^{\sigma+2}  f\|_{L^2}
$$
and $\mathcal Z_\varphi^6$ is the space induced by the norm
$$
\|\an{x}^{-2}(1+\varphi)^{-1/3} f\|_{L^6}.
$$

We proceed by duality. Let $g$ be in the dual of $\mathcal Z_\varphi^2$, that is 
$$
\|(1+\varphi)\an{x}^2 D^{-2-\sigma} g\|_{L^2} < \infty .
$$
We estimate $\an{g, w*u}$ where $\an{\cdot ,\cdot}$ is the inner product. We have $\an{g,w*u} = \an{w_1*g,u}$ with $w_1(x) = w(-x)$. For all $t \in \R$, we have 
$$
\an{g,w*u} = \an{ \an{x}^2 D^{-s} S(t) (w_1 * g),\an{x}^{-2} D^s S(t) u }.
$$
And hence we get for all $t$,
$$
|\an{g,w*u}| \leq \|\an{x}^2 D^{-s} S(t) (w_1 * g)\|_{L^2_x}\|\an{x}^{-2} D^s S(t) u\|_{L^2_x}
$$
and as the left hand side does not depend on $t$, we can take the $L^2$ norm in time between $0$ and $1$ to get
$$
|\an{g,w*u}| \leq \sup_{t\in [0,1]}\Big(\|\an{x}^2 D^{-s} S(t) (w_1 * g)\|_{L^2_x}\Big)\|\,\|\an{x}^{-2} D^s S(t) u\|_{L^2_x}\|_{L^2(t\in[0,1])}
$$
which yields
$$
|\an{g,w*u}| \leq \sup_{t\in [0,1]}\|\an{x}^2 D^{-s} S(t) (w_1 * g)\|_{L^2_x}p_s( u).
$$

We estimate $\|\an{x}^2 D^{-s} S(t) (w_1 * g)\|_{L^2_x}$. We consider the Fourier transform to get
$$
\|\an{x}^2 D^{-s} S(t) (w_1 * g)\|_{L^2_x} = \|D_k^2 \an{k}^{-s} e^{-ik^2 t} \hat w_1(k) \hat g(k)\|_{L^2_x}= \|D_k^2 \an{k}^{-s+2+\sigma} e^{-ik^2 t} \hat w_1(k) \an{k}^{-2-\sigma}\hat g(k)\|_{L^2_x}.
$$
We distribute $D_k^2$ to get
\begin{multline*}
\|\an{x}^2 D^{-s} S(t) (w_1 * g)\|_{L^2_x} \leq \|D_k^2 \an{k}^{-s+2+\sigma} e^{-ik^2 t} \hat w_1(k)\|_{L^\infty_k} \|\an{k}^{-2-\sigma}\hat g(k)\|_{L^2_x} + \\
\|\an{k}^{-s+2+\sigma} e^{-ik^2 t} \hat w_1(k)\|_{L^\infty_k}  \|D_k^2\an{k}^{-2-\sigma}\hat g(k)\|_{L^2_x}.
\end{multline*}

We have 
$$
 \|\an{k}^{-2-\sigma}\hat g(k)\|_{L^2_x} = \| D^{-2-\sigma} g\|_{L^2} \leq \|(1+\varphi) \an{x}^2 D^{-2-\sigma} g\|_{L^2}
 $$
and  
$$
\|D_k^2\an{k}^{-2-\sigma}\hat g(k)\|_{L^2_x} = \|\an{x}^2 D^{-2-\sigma} g\|_{L^2} \leq \|(1+\varphi) \an{x}^2 D^{-2-\sigma} g\|_{L^2}.
$$

Regarding $w$, we see that for $t\in [0,1]$,
\begin{eqnarray*}
\|D_k^2 \an{k}^{-s+2+\sigma} e^{-ik^2 t} \hat w_1(k)\|_{L^\infty_k}& \leq &\|\an{k}^{-s+4 + \sigma}\hat w_1(k)\|_{L^\infty_k}+ \|\an{k}^{-s+2+\sigma} D_k^2 \hat w_1 \|_{L^\infty_k}\\
\|\an{k}^{-s+2+\sigma} e^{-ik^2 t} \hat w_1(k)\|_{L^\infty_k}  & \leq & \|\an{k}^{-s+2+\sigma} \hat w_1(k)\|_{L^\infty_k} 
\end{eqnarray*}
and taking the inverse Fourier transform
\begin{eqnarray*}
\|\an{k}^{-s+4 + \sigma}\hat w_1(k)\|_{L^\infty_k} & \leq & \|D_x^{-s+4 + \sigma} w_1\|_{L^1_x}\\
\|\an{k}^{-s+2+\sigma} D_k^2 \hat w_1 \|_{L^\infty_k} & \leq & \|D_x^{-s+2+\sigma}\an{x}^2 w_1 \|_{L^1_x} \\
\|\an{k}^{-s+2+\sigma} \hat w_1(k)\|_{L^\infty_k} &\leq & \|D_x^{-s+2 + \sigma} w_1\|_{L^1_x}.
\end{eqnarray*}
Given that $w_1$ is $\mathcal C^\infty$ with compact support, all these quantities are finite and 
$$
|\an{g,w*u}|\leq C(w) \|\an{x}^2 D^{-2-\sigma} g\|_{L^2} p_s( u).
$$

Therefore, as it is true for all $g$ in the dual of $\mathcal Z_\varphi^2$,
$$
\|w*u\|_{\mathcal Z_\varphi^2} \leq C(w) p_s(u).
$$

The same proof applies for $\mathcal Z_\varphi^6$.
\end{proof}

\begin{proposition}The measure $\rho_N$ is invariant under the flow $\psi_N$ for the topological $\sigma$-algebra of $\mathcal Z_\varphi$. \end{proposition}

\begin{proof} Let $F$ be a bounded measurable function on $\mathcal Z_\varphi$. As for $\rho_N$ almost all $u$, $\psi_N(t) u$ belongs to $\mathcal Z_\varphi$ (Lemma \ref{lem-welldef}), and since $\rho_N$ is defined on $\mathcal Z_\varphi$, we have that $\E_{\rho_N} (F\circ \psi_N(t))$ is well-defined. 

Let $w_k$ be a sequence of $\mathcal C^\infty $ functions with compact supports which converges towards a Dirac delta. Let $F_k : u\mapsto F(w_k * u)$. Thanks to Lemma \ref{lem-meas}, we have that $u\mapsto w_k*u$ is continuous and hence measurable from $\mathcal Y_s$ to $\mathcal Z_\varphi$ and thus $F_k$ is measurable and bounded on $\mathcal Y_s$. We deduce 
$$
\E_{\rho_N}(F_k \circ \psi_N(t)) = \E_{\rho_N}(F_k ).
$$
As $\psi_N(t) u$ belongs almost surely to $\mathcal Z_\varphi$ and $F$ is bounded, we can apply the dominated convergence theorem to pass to the limit when $k\rightarrow \infty$, which yields
$$
\E_{\rho_N}(F \circ \psi_N(t)) = \E_{\rho_N}(F )
$$
for all $t$ and concludes the proof. \end{proof}

\section{Proof of the theorem}\label{proofs}

Before we start applying the results of the last section to prove the theorem, we state two useful and central lemmas.

\subsection{Two technical results}\label{lemmas}

\begin{lemma}\label{lem-magicwithoutder} Let $r\geq  1$, there exist a non-negative, even and increasing on $\R^+$ function $\varphi_r$ such that for all $x\in \R$ and all $N \in \N$, we have 
$$
\E_{\mu_N}(|u(x)|^r) \leq \varphi_r(x) \, ,\, \E_{\rho}(|u(x)|^r) \leq \varphi_r(x) .
$$
\end{lemma}

\begin{remark} This result may be seen as a consequence of an estimate on the ground state $\Omega_V$ of $L$, \cite{D90}, or as a consequence of a Brascamp-Lieb inequality, as in \cite{B00}. \end{remark}

\begin{proof} Let $x\in \R$ and let $N\geq |x|$. We apply Theorem 6.7 in \cite{SIMfunint} page 57 with 
$$
G(u) = |u(x)|^r \Omega_0(u(-N)) \Omega_V^{-1}(u(-N))\Omega_0(u(N)) \Omega_V^{-1}(u(N))e^{-2E(V) N}
$$
where we recall that $\Omega_V$ is the eigenstate associated to the non-degenerate first eigenvalue $E(V)$ of $L= -\frac12 \lap_u + |u|^2+|u|^4-\frac12$ and $\Omega_0$ is the eigenstate associated to the non-degenerate first eigenvalue $0$ of $L_0 = -\frac12\lap_u + |u|^2-\frac12$. We get on the one hand
$$
\int G(u) \Omega_V(u(-N)) \Omega_0^{-1}(u(-N))\Omega_V(u(N)) \Omega_0^{-1}(u(N))e^{2E(V) N} d\mu_N(u) = \int |u(x)|^r d\mu_N(u)
$$
and on the other hand
\begin{multline*}
\int G(u) \Omega_V(u(-N)) \Omega_0^{-1}(u(-N))\Omega_V(u(N)) \Omega_0^{-1}(u(N))e^{2E(V) N} d\mu_N(u) =\\
\int \tilde G(u_{-N},u_x,u_N)
\Omega_V(u_N) \Omega_V(u_{-N}) e^{-(x+N)\hat L} (u_{-N},u_x) e^{-(N-x)\hat L} (u_x,u_N) du_{-N}du_x du_N
\end{multline*}
with 
$$ \tilde G(u_{-N,u_x,u_N)}=
|u_x|^r \Omega_0(u_{-N}) \Omega_V^{-1}(u_{-N})\Omega_0(u_N) \Omega_V^{-1}(u_N)e^{-2E(V) N} .
$$
We recall that $e^{-s\hat L}(u_1,u_2)$ is the fundamental solution to $\partial_s y= -\hat L y$, that is
$$
y(s,u_2) = \int du_1 e^{-s\hat L}(u_1,u_2)y(0,u_1)
$$
and $\hat L = L - E(V)$. 

By simplifying the $\Omega_V$ we get
$$
\int |u(x)|^r d\mu_N(u) = \int |u_x|^r  \Omega_0(u_{-N}) \Omega_0(u_N)e^{-2E(V) N} e^{-(x+N)\hat L} (u_{-N},u_x) e^{-(N-x)\hat L} (u_x,u_N) du_{-N}du_x du_N.
$$

Let $\hat L_0 = L_0 - E(V)$, we have $\hat L - \hat L_0 = |x|^4$, thus by the maximum principle, we get $e^{-s\hat L}(u_1,u_2) \leq e^{-s \hat L_0} (u_1,u_2)$. Therefore
\begin{multline*}
\int |u(x)|^r d\mu_N(u) \leq \\
\int |u_x|^r  \Omega_0(u_{-N}) \Omega_0(u_N)e^{-2E(V) N} e^{-(x+N)\hat L_0} (u_{-N},u_x) e^{-(N-x)\hat L_0} (u_x,u_N) du_{-N}du_x du_N.
\end{multline*}
By definition of $\Omega_0$, we have that 
$$
\int du_1 \Omega_0(u_1) e^{-s\hat L_0}(u_1,u_2) = e^{sE(V)} \Omega_0 (u_2) = \int du_1 \Omega_0(u_1) e^{-s\hat L_0}(u_2,u_1).
$$
Hence integrating over $u_N$ and $u_{-N}$ yields
$$
\int |u(x)|^r d\mu_N(u) \leq \int |u_x|^r  \Omega_0(u_x)^2 e^{-2E(V) N} e^{(x+N)E(V)}  e^{(N-x)E(V)} du_x 
$$
and thus
$$
\int |u(x)|^r d\mu_N(u) \leq \int |u_x|^r  \Omega_0(u_x)^2  du_x .
$$
We have that $\Omega_0(u)$ behaves as $e^{-c|u|^2}$ hence the above quantity is finite. Therefore, there exists a constant, depending only on $r$, $C_r$ such that for all $N\geq |x|$, 
$$
\int |u(x)|^r d\mu_N(u) \leq C_r .
$$
Let 
$$
\varphi_r(x) = \max_{N< |x|} \E_{\mu_N}(|u(x)|^r) \leq C_r \max_{N< |x|} D_N^{-1} < \infty.
$$
We have that $\varphi_r $ is a non negative, increasing function.

Finally, from Theorem 6.9 in \cite{SIMfunint} page 58, we get
$$
\E_\rho(|u(x)|^r) \leq C_r \leq \varphi_r(x).
$$
\end{proof}

We now include derivatives in our analysis.

\begin{lemma}\label{lem-magicwithder}Let $r\geq  2$ and $\alpha < \frac12$ such that $0\leq \alpha \leq \min(\frac2{r}, 1- \frac3{2r})$, there exist a non-negative, increasing on $\R^+$, even function $\varphi_{\alpha,r}$ such that for all $x,y\in \R$, $|x|\geq |y|$ and all $N \in \N$, we have 
$$
\E_{\mu_N}\Big(\frac{|u(x)-u(y)|^r}{|x-y|^{1+\alpha r}}\Big) \leq \varphi_{\alpha ,r}(x)  \, ,\, \E_{\rho}\Big(|\frac{|u(x)-u(y)|^r}{|x-y|^{1+\alpha r}}\Big) \leq \varphi_{\alpha ,r}(x) .
$$
\end{lemma}

\begin{proof} We essentially use the same method as previously. Let $x,y \in \R$ and let $N \geq \max (|x|,|y|)$. We assume, without loss of generality, $x\geq y$. We apply Theorem 6.7 in \cite{SIMfunint} page 57 with 
$$
G(u) = \frac{|u(x)-u(y)|^r}{|x-y|^{1+\alpha r}}\Omega_0(u(-N)) \Omega_V^{-1}(u(-N))\Omega_0(u(N)) \Omega_V^{-1}(u(N))e^{-2E(V) N}
$$
We get on the one hand
$$
\int G(u) \Omega_V(u(-N)) \Omega_0^{-1}(u(-N))\Omega_V(u(N)) \Omega_0^{-1}(u(N))e^{2E(V) N} d\mu_N(u) = \int \frac{|u(x)-u(y)|^r}{|x-y|^{\alpha r+1}} d\mu_N(u)
$$
and on the other hand
\begin{multline*}
\int G(u) \Omega_V(u(-N)) \Omega_0^{-1}(u(-N))\Omega_V(u(N)) \Omega_0^{-1}(u(N))e^{2E(V) N} d\mu_N(u) = \\
\int \frac{|u_x-u_y|^r}{|x-y|^{\alpha r+1}} \Omega_0(u_{-N}) \Omega_V^{-1}(u_{-N})\Omega_0(u_N) \Omega_V^{-1}(u_N)e^{-2E(V) N} \\
\Omega_V(u_N) \Omega_V(u_{-N}) e^{-(y+N)\hat L} (u_{-N},u_y) e^{-(x-y)\hat L} (u_y,u_x) e^{-(N-x)\hat L} (u_x,u_N) du_{-N}du_ydu_x du_N
\end{multline*}

By simplifying the $\Omega_V$ we get
\begin{multline*}
\int \frac{|u(x)-u(y)|^r}{|x-y|^{\alpha r +1}} d\mu_N(u) = \\
\int  \tilde G
e^{-(y+N)\hat L} (u_{-N},u_y) e^{-(x-y)\hat L} (u_y,u_x) e^{-(N-x)\hat L} (u_x,u_N) du_{-N}du_ydu_x du_N
\end{multline*}
with
$$
\tilde G(u_{-N},u_y,u_x,u_N) =
\frac{|u_x-u_y|^r}{|x-y|^{\alpha r +1}} \Omega_0(u_{-N}) \Omega_0(u_N) e^{-2E(V) N} .
$$

Using as previously the maximum principle, we get
\begin{multline*}
\int \frac{|u(x)-u(y)|^r}{|x-y|^{\alpha r +1}} d\mu_N(u) \leq
\int \frac{|u_x-u_y|^r}{|x-y|^{\alpha r +1}} \Omega_0(u_{-N}) \Omega_0(u_N) e^{-2E(V) N}  e^{-(y+N)\hat L_0} (u_{-N},u_y) \\
e^{-(x-y)\hat L_0} (u_y,u_x) e^{-(N-x)\hat L_0} (u_x,u_N) du_{-N}du_ydu_x du_N.
\end{multline*}

Integrating over $u_{-N}$ and $u_N$ yields
$$
\int \frac{|u(x)-u(y)|^r}{|x-y|^{\alpha r +1}} d\mu_N(u) = \int \frac{|u_x-u_y|^r}{|x-y|^{\alpha r +1}} \Omega_0(u_x) \Omega_0(u_y)    e^{-(x-y) L_0} (u_y,u_x) du_ydu_x .
$$
We remark that the $\hat L_0$ as turned into $L_0$ as we simplified with $e^{-2E(V) N}$.

When $\alpha r \leq 1$, we use the estimates \eqref{estimQ1}, \eqref{estimQ2} and the fact that the derivative at $z=0$ of $e^{-zL_0}(u_1,u_2)$ is $0$ outside the diagonal $u_1=u_2$ to get
$$
e^{-(x-y) L_0} (u_y,u_x) |x-y|^{-1-\alpha r} \lesssim (1+ |u_x-u_y|^{-3/2-\alpha r}|u_x+u_y|^{1+\alpha r })
$$
We get 
$$
\int \frac{|u(x)-u(y)|^r}{|x-y|^{1+\alpha r}} d\mu_N(u) \leq C_{r,\alpha} \int |u_x-u_y|^{r-3/2-\alpha r} \Omega_0(u_x) \Omega_0(u_y)  |u_x+u_y|^{\alpha r+1 }   du_ydu_x .
$$
With the choice of $\alpha$, $r-3/2-\alpha r$ is non-negative, and since $\Omega_0(u)$ behaves like $e^{-c|u|^2}$, the above quantity is finite and does not depend on $x$ or $y$. Hence, there exists $C_{r,\alpha}$ such that for all $N \geq \max(|x|,|y|)$,
$$
\int \frac{|u(x)-u(y)|^r}{|x-y|^{1+\alpha r}} d\mu_N(u) \leq C_{r,\alpha}.
$$

When $1\leq \alpha r \leq 2$, we use the estimates \eqref{estimQ2}, \eqref{estimQ3} and the fact that the two first derivatives at $z=0$ of $e^{-zL_0}(u_1,u_2)$ are $0$ outside the diagonal $u_1=u_2$ to get
$$
e^{-(x-y) L_0} (u_y,u_x) |x-y|^{-1-\alpha r} \lesssim (1+ |u_x-u_y|^{-3/2-\alpha r}|u_x+u_y|^{1+\alpha r })
$$
We get 
$$
\int \frac{|u(x)-u(y)|^r}{|x-y|^{1+\alpha r}} d\mu_N(u) \leq C_{r,\alpha} \int |u_x-u_y|^{r-3/2-\alpha r} \Omega_0(u_x) \Omega_0(u_y)  |u_x+u_y|^{\alpha r+1 }   du_ydu_x .
$$
With the choice of $\alpha$, $r-3/2-\alpha r$ is non-negative, and since $\Omega_0(u)$ behaves like $e^{-c|u|^2}$, the above quantity is finite and does not depend on $x$ or $y$. Hence, there exists $C_{r,\alpha}$ such that for all $N \geq \max(|x|,|y|)$,
$$
\int \frac{|u(x)-u(y)|^r}{|x-y|^{1+ \alpha r}} d\mu_N(u) \leq C_{r,\alpha}.
$$

For $N\leq \max(|x|,|y|)$, we have
$$
\int \frac{|u(x)-u(y)|^r}{|x-y|^{\alpha r +1}} d\mu_N(u) \leq D_N^{-1} \E\Big(\frac{|q(x)-q(y)|^r}{|x-y|^{\alpha r +1}}\Big).
$$
As $\alpha < \frac12$, we get that the mean value on the oscillator process is finite. Let
$$
C_{r,\alpha}'= \Big( \min\Big( \E\Big(\frac{|q(x)-q(y)|^r}{|x-y|^{\alpha r +1}}\Big),C_{r,\alpha}\Big)\Big)
$$
and 
$$
\varphi_{\alpha,r} (x) = C_{r,\alpha}'\max_{N\leq |x|} D_N^{-1}.
$$
For all $N$, $D_N \leq 1$, hence for all $N$,
$$
\E_{\mu_N}\Big(\frac{|u(x)-u(y)|^r}{|x-y|^{\alpha r +1}}\Big) \leq \varphi_{\alpha ,r}(\max(|x|,|y|)) .
$$

Finally, from Theorem 6.9 in \cite{SIMfunint} page 58, we get
$$
\E_{\rho}\Big(|\frac{|u(x)-u(y)|^r}{|x-y|^{\alpha r +1}}\Big) \leq \varphi_{\alpha ,r}(\max(|x|,|y|)) .
$$
\end{proof}

\subsection{Convergence of \texorpdfstring{$\rho_N$}{rhoN} towards \texorpdfstring{$\rho$}{rho}}\label{conv}

In this subsection, we prove that the sequence $\rho_N$ converges towards $\rho$ in law.

\begin{proposition}\label{prop-convmeas} For all non-negative increasing function $\varphi$, we have that the sequence $\rho_N$ converges towards $\rho$ in law in the sense that for all bounded measurable function $F$ from $\mXe_\varphi$ to $\R$, the sequence $\E_{\rho_N}(F)$ converges towards $\E_\rho(F)$.\end{proposition}

\begin{lemma}\label{lem-boundDN} Let $F$ be a measurable function from $\mXe_\varphi$ to $\R$ such that either $F$ is bounded or there exists $x_0 \in \R$ and $r\geq 1$ satisfying $F(u) = |u(x_0)|^r$ or there exists $x_0$ and $y_0$ in $\R$, $r \geq2$ and $\alpha \in [0,\frac12[$ such that $\alpha \leq \frac{2}{r},1-\frac{3}{2r}$ satisfying $F(u) = \frac{|u(x_0)-u(y_0)|^r}{|x_0-y_0|^{\alpha r+1}}$. We have 
$$
|\E_{\rho_N}(F) - \E_{\mu_N}(F)|\leq C_F \frac{D_N}{1-D_N^2}
$$
where $C_F$ depends on $F$ in the case $F$ bounded and on $r$, or $r$ and $\alpha$, in the other cases but not on $N$.
\end{lemma}

\begin{proof} Let $g_N(u) = e^{-\frac12\int\chi_N(x) |u(x)|^4dx}$, $h_N(x) = e^{-\frac12\int_{-N}^N |u(x)|^4dx}$ and set $D_N' = \int g_N(u) dq(u)$. We recall that $D_N = \int h_N(u )dq(u)$.

Given Proposition \ref{ovvio}, we have 
$$
\E_{\rho_N}(F) = \int F(u) \frac{g_N(u)}{D'_N} dq(u) .
$$
By definition, we have 
$$
\E_{\mu_N}(F) = \int F(u) \frac{h_N(u)}{D_N} dq(u) .
$$

Let us prove that
$$
\int |F(u)|\,|g_N(u) - h_N(u)|dq(u) \leq C D_N^3.
$$
Indeed, as $\chi_N(x) \in [0,1]$ and $\chi_N(x) = 1$ on $[-N,N]$, we have $g_N(u) \leq h_N(u)$ and thus 
$$
|g_N(u) - h_N(u)| \leq h_N(u) \frac12\int |\chi_N(x) - 1_{[-N,N]}(x)| |u(x)|^4dx.
$$
Integrating with respect to $dq(u)$ and bounding $h_N(u)$ by $1$ yields in the case that $F(u)$ is bounded
$$
\int |F(u)|\,|g_N(u) - h_N(u)|dq(u) \leq C_F\int |\chi_N(x) - 1_{[-N,N]}(x)| \Big( \int |u(x)|^4 dq(u) \Big)dx.
$$

In the cases $F(u) = |u(x_0)|^r$ or $F(u) = \frac{|u(x_0)-u(y_0)|^r}{|x_0-y_0|^{1+\alpha r}}$, we get
$$
\int |F(u)|\,|g_N(u) - h_N(u)|dq(u) \leq \int |\chi_N(x) - 1_{[-N,N]}(x)| \Big( \int |F(u)| \, |u(x)|^4 dq(u) \Big)dx.
$$

Since $dq$ is up to a constant the Gaussian law induced by the random variable 
$$
\int \frac{e^{inx}}{\sqrt{1+n^2}} dW(n)
$$
where $W(n)$ is the reunion of two independent complex Brownian motions, we get
$$
\int |u(x)|^4 dq(u) \leq C \Big( \int |u(x)|^2 dq(u) \Big)^2
$$
where $C$ is a universal constant related to Gaussian variables. We have $\int |u(x)|^2 dq(u) = \int \frac{dn}{1+n^2}= \pi$, thus
$$
\int |u(x)|^4 dq(u) \leq C
$$
where $C$ does not depend on $x$. 

We use the proof of Lemma \ref{lem-magicwithoutder} in the case $F(u) = |u(x_0)|^r$ and the proof of Lemma \ref{lem-magicwithder} in the case $F(u) = \frac{|u(x_0)-u(y_0)|^r}{|x_0-y_0|^{1+\alpha r}}$ to get that
$$
 \int |F(u)|\, |u(x)|^4 dq(u) 
 $$
 is finite, depends on $r$ or $r$ and $\alpha$ but is independent from $x$ and $x_0$ or $x$, $x_0$, and $y_0$.

We get
$$
\int |F(u)|\, |g_N(u) - h_N(u)|dq(u) \leq C_F\frac12\int |\chi_N(x) - 1_{[-N,N]}(x)|dx .
$$
By definition of $\chi_N$, it is equal to $1$ on $[-N,N]$, to $0$ outside $[-N-D_N^3,N+D_N^3]$ and belongs to $[0,1]$, hence 
$$
\frac12 \int |\chi_N(x) - 1_{[-N,N]}(x)| \leq D_N^3
$$
and 
$$
\int |F(u)|\, |g_N(u) - h_N(u)|dq(u) \leq C_F D_N^3.
$$

We deduce from that $|D_N' - D_N| \leq D_N^3$ for $F = 1$ and thus $D_N' \geq D_N-D_N^3$.

We have 
$$
|\E_{\rho_N}(F) - \E_{\mu_N}(F)| \leq \int |F(u)| \frac{|g_N(u)-h_N(u)|}{D_N'} dq(u) + \int |F(u)| h_N(u) \frac{|D_N - D_N'|}{D_ND_N'}.
$$
Given the previous estimates, we have 
$$
|\E_{\rho_N}(F) - \E_{\mu_N}(F)| \leq C_F \frac{D_N^3}{D_N'} + \int h_N(u)|F(u)| dq(u) \frac{D_N^3}{D_ND_N'}.
$$
We get by bounding $h_n$ by $1$
$$
|\E_{\rho_N}(F) - \E_{\mu_N}(F)| \leq C_F \frac{D_N^2}{D_N'} \leq C_F \frac{D_N}{1-D_N^2}
$$
which concludes the proof.\end{proof}

\begin{proof}[Proof of Proposition \ref{prop-convmeas}.] As $\mu_N $ converges in law towards $\rho$ for cylindrical sets, we get that $\mu_N$ converges towards $\rho$ for the topological $\sigma$ algebra of continuous functions $u$ such that $\an{x}^{-\nu}u$ belongs to $L^\infty$ for any $\nu >0$.

Indeed, let $\overline B(u_0, R)$ the closed ball of centre $u_0$ and radius $R$ in this space, we have 
$$
\overline B(u_0,R) = \{ u \in \mathcal C(\R,\C) \, | \, \forall x\in \R , |u(x) - u_0(x)|\leq R \an{x}^\nu\}.
$$
As the $u$ are continuous, we can restrict $x$ to $\mathbb Q$ and get
$$
\overline{B}(u_0,R) = \bigcap_{x\in \mathbb Q} \{u \in \mathcal C(\R,\C) \, | \,  |u(x) - u_0(x)|\leq R \an{x}^\nu\}.
$$
As $\{u \in \mathcal C(\R,\C) \, | \,  |u(x) - u_0(x)|\leq R \an{x}^\nu\}$ is a cylindrical set, we get the convergence of $\mu_N$ towards $\rho$ for the balls of continuous functions in the norm $\|\an{x}^{-\nu}\cdot\|_{L^\infty}$ and as these balls generates the topological $\sigma $ algebra we get the convergence of $\mu_N$ towards $\rho$.

For this to be significant, we prove that 
$$
\rho(\{ u \in \mathcal C(\R,\C) \, |\, \|\an{x}^{-\nu}u\|_{L^\infty}<\infty\}) = 1
$$
and
$$
\mu_N(\{ u \in \mathcal C(\R,\C) \, |\, \|\an{x}^{-\nu}u\|_{L^\infty}<\infty\}) = 1.
$$
Indeed, if we do not have these properties then the convergence in $\an{x}^\nu L^\infty$ is only true on the set 
$$
\{ u \in \mathcal C(\R,\C) \, |\, \|\an{x}^{-\nu}u\|_{L^\infty}<\infty\}
$$ 
which has not a full $\rho$ or $\mu_N$ measure and hence one cannot have the convergence in law.

Given that the $\mu_N$ are absolutely continuous with respect to $q$, it is enough to prove that $u$ is $dq$ almost surely continuous and such that $\an{x}^{-\nu}u$ belongs to $L^\infty$. 

Let $\alpha \in [0,\frac12[$ and $p\in ]1,\infty[$ such that $\frac1{p} < \min (\nu, \frac12 -\alpha)$, we have by Sobolev inequality, for $\frac1{p}<s<\frac12 -\alpha$,
$$
\|\an{x}^{-\nu}D^\alpha u\|_{L^p_q L^\infty(\R)} \leq C \|D^s\an{x}^{-\nu}D^\alpha u\|_{L^p_q L^p(\R)}.
$$
As differentiating $\an{x}^{-\nu}$ only gains in powers of $x$ and since we can reverse the order of integration we get,
$$
\|\an{x}^{-\nu}D^\alpha u\|_{L^p_q( L^\infty(\R)} \leq C \|\an{x}^{-\nu}D^{s+\alpha} u\|_{ L^p(\R, L^p_q)}.
$$
Since $s+\alpha < \frac12$ and $dq$ is a Gaussian, we have 
$$
\|D^{s+\alpha}u(x)\|_{L^p_q} \leq C_p \|D^{s+\alpha}u(x)\|_{L^2(dq)} \leq C_p \Big( \int \frac{dn}{(1+n^2)^{1-(s+\alpha)}} \Big)^{1/2}<\infty.
$$
And since $\an{x}^{-\nu}$ belongs to $L^p$, we get
$$
\|\an{x}^{-\nu}D^\alpha u\|_{L^p_q L^\infty(\R)} < \infty
$$
which yields that $ \|\an{x}^{-\nu}D^\alpha u\|_{L^\infty(\R)}$ is $\mu_N$ almost surely finite and hence $\an{x}^{-\nu} u$ belongs $\mu_N$ almost surely to $W^{\alpha, \infty}$ which ensures that $u$ is $\mu_N$ almost surely continuous and that $\an{x}^{-\nu} u$ belongs $\mu_N$ almost surely to $L^\infty$.

For $\rho$, we use Theorem 6.9 in \cite{SIMfunint} to get that for $r>4$ and $\alpha \in ]\frac1{r},\frac2{r}[$ (the couple $(r,\alpha)$ satisfies the assumption in Lemma \ref{lem-magicwithder}), we have that
$$
\E_{\rho}\Big( \frac{|u(x)-u(y)|^r}{|x-y|^{1+\alpha r}}\Big)
$$
is bounded uniformly in $x$ and $y$. Hence for $\nu > \frac1{r}$ we get that
$$
\E_{\rho}  \Big(\int dx \int dy \an{x}^{-\nu r} \an{y}^{-\nu r} \frac{|u(x)-u(y)|^r}{|x-y|^{1+\alpha r}} \Big)
$$
is finite and hence $\an{x}^{-\nu} u$ belongs $\rho$ almost surely to $W^{\alpha, r}$ which ensures that $u$ is $\rho$ almost surely continuous and that $\an{x}^{-\nu} u$ belongs $\rho$ almost surely to $L^\infty$.

The topology of $\mXe_\varphi$ is weaker that the topology of $\an{x}^\nu L^\infty$. Indeed,
$$
\|u\|_{\mXe_\varphi} = \|(1+\varphi)^{-1}\an{x}^{-6(1+\varepsilon)}D^{\sigma(1+\varepsilon)} u\|_{L^2} \leq \|D^{-\sigma(1+\varepsilon)}\an{x}^{-6(1+\varepsilon)}D^{\sigma(1+\varepsilon)} u\|_{L^2}.
$$
We recall that $\sigma < 0$ and that differentiating $\an{x}^{-6(1+\varepsilon)}$ only gains in powers of $x$ thus
$$
\|u\|_{\mXe_\varphi}\leq C \|\an{x}^{-6(1+\varepsilon)} u\|_{L^2}
$$
and by H\"older inequality
$$
\|u\|_{\mXe_\varphi}\leq C \|\an{x}^{-\nu} u\|_{L^\infty}.
$$

Thus, we get that $\mu_N$ converges towards $\rho$ in law for the topological $\sigma$ algebra of $\mXe_\varphi$. This implies that for all $F$ measurable, bounded from $\mXe_\varphi$ to $\R$, we have that $\E_{\mu_N}(F)$ converges towards $\E_\rho(F)$. Given the lemma, we get that $\E_{\rho_N}(F)$ converges towards $\E_\rho(F)$ which implies the result. \end{proof}

We deduce from Lemmas \ref{lem-magicwithoutder}, \ref{lem-magicwithder} and \ref{lem-boundDN} the following lemma.

\begin{lemma}\label{lem-theoneforfed} Let $r\geq 2$ and $\alpha \in [0,\frac12[$ such that $\alpha r \leq 2,1- \frac3{2r}$, there exists $\varphi_r$ and $\varphi_{r,\alpha}$ two non negative increasing functions such that for all $N\in \N$ and all $x,y\in \R$, $|x|\geq |y|$ we have
\begin{equation}\label{phi1}
\E_{\rho_N} (|u(x)|^r) \leq \varphi_r(x)
\end{equation}
and 
\begin{equation}\label{phi2}
\E_{\rho_N}\Big( \frac{|u(x)-u(y)|^r}{|x-y|^{1+\alpha r}} \Big) \leq \varphi_{r,\alpha}(x) .
\end{equation}
\end{lemma}

\subsection{Tightness of \texorpdfstring{$\nu_N$}{nuN}}\label{subsec-tightness}\label{tightsec}

This section is devoted to prove the tightness of the family of measures $(\nu_N)_N$. In order to do this, we need to begin by proving some preliminary technical results.  We begin with the following compactness argument.

\begin{proposition}\label{compact}
Let $R_0>0$. The set $K=\{u:\|u\|_{\mathcal{X}_\varphi}\leq R_0\}$ is compact in $\mathcal{X}^\varepsilon_\varphi$.
\end{proposition}

\begin{proof}
We show that for every $\varepsilon>0$ there exists $n_\varepsilon$ and $u_1,\dots u_{n_\varepsilon}$ such that
$$
K\subset \bigcup_{j=1}^{n_\varepsilon}B(u_j,\varepsilon)
$$
where $B$ are the balls in the $\mathcal{X}^\varepsilon_\varphi$ topology. To do that, we introduce a smooth cut-off function $1_{|x|\leq R}$ such that $1_{|x|\leq R}(x)=1$ for $x\in[-R,R]$ and $1_{|x|\leq R}(x)=0$ for $x\in(-\infty,-2R)\cup(2R,+\infty)$. We then have, for any $u\in K$, 
\begin{equation*}
\|u\|_{\mathcal{X}^\varepsilon_\varphi}\leq
I+II
\end{equation*}
where
$$
I=\|(1+\varphi)^{-1}\langle x\rangle^{-6(1+\varepsilon)}D^{\sigma(1+\varepsilon)}(1-1_{|x|\leq R})u\|_{L^2}
$$
and
$$
II=\|(1+\varphi)^{-1}\langle x\rangle^{-6(1+\varepsilon)}D^{\sigma(1+\varepsilon)}1_{|x|\leq R}u\|_{L^2}
$$

The first term is easily bounded as follows
\begin{equation}\label{I}
I\leq C R^{-\varepsilon}\|(1+\varphi)^{-1}\langle x\rangle^{-6}D^{\sigma(1+\varepsilon)} u\|_{L^2}\leq CR^{-\varepsilon} R_0.
\end{equation}
To estimate the second term, we need to introduce also a frequency cut-off $\Pi_N$
$$
\widehat{\Pi_N f}(n) = \eta \Big(\frac{n}{N}\Big) \hat f (n)
$$
with $\eta$ a non negative even $\mathcal C^\infty $ function with compact support included in $[-1,1]$ and such that $\eta = 1$ on $[-1/2,1/2]$ and $N>0$. We thus rewrite
\begin{equation*}
II=II_A+II_B
\end{equation*}
where
$$
II_A=\|(1+\varphi)^{-1}\langle x\rangle^{-6(1+\varepsilon)}D^{\sigma(1+\varepsilon)}(1-\Pi_N)1_{|x|\leq R}u\|_{L^2}
$$
and
$$
II_B=\|(1+\varphi)^{-1}\langle x\rangle^{-6(1+\varepsilon)}D^{\sigma(1+\varepsilon)}\Pi_N1_{|x|\leq R}u\|_{L^2}
$$

To estimate $II_A$ we use the fact that $\Pi_N$ cuts off high frequencies, yielding
\begin{eqnarray}\label{II}
II_A&\leq& C_R N^{-\varepsilon}\|(1+\varphi)^{-1}\langle x\rangle^{-6}D^{\sigma}1_{|x|\leq R}u\|_{L^2}
\\
\nonumber
&\leq&
C_R N^{-\varepsilon} R_0.
\end{eqnarray}
Finally, to estimate $II_B$ we use that $\Pi_N1_{|x|\leq R}u$ is finite dimensional, and therefore for every $\varepsilon>0$ there exist $n_\varepsilon$ and $u_1,\dots u_{n_\varepsilon}\in \Pi_N1_{|x|\leq R}K$ such that
\begin{equation}\label{III}
\displaystyle
\Pi_N 1_{|x|\leq R}K\subset \bigcup_{j=1}^{n_\varepsilon}B(u_j,\varepsilon/3).
\end{equation}
We are now ready to conclude: for a fixed $\varepsilon>0$, we can choose $R$ in \eqref{I} big enough such that $I\leq \varepsilon/3$ and, for fixed $\varepsilon$ and $R$, we can choose $N$ in \eqref{II} big enough such that $II_A\leq \varepsilon/3$. Therefore, taking any $u\in K$, we can conclude that there exist $j\in\{1,\dots n_\varepsilon\}$ such that  taking the corresponding $u_j\in \mathcal{X}_\varepsilon$ in $II_B$ gives
$$
\|u-u_j\|_{\mathcal{X}^\varepsilon_\varphi}\leq \frac23\varepsilon+ \|\Pi_N1_{|x|\leq R}u-u_j\|_{\mathcal{X}^\varepsilon_\varphi}\leq \varepsilon
$$
and thus the proof is concluded.
\end{proof}

As a consequence, we have the following
\begin{corollary}\label{compeb}
For every $\varepsilon>0$ the embedding $\mathcal{X}_\varphi\subset \mathcal{X}^\varepsilon_\varphi$ is compact.
\end{corollary}

Another crucial tool is represented by the following uniform estimates.

\begin{proposition}\label{propest}
Let $0<\alpha\leq 1/2$. Then there exists a non negative increasing function $\varphi(x)$ such that
\begin{equation}\label{crucialest}
\|u\|_{L^2_{\nu_N},\mathcal{X}_{T,\varphi}}
\end{equation}
is uniformly bounded in $N$.

\end{proposition}

We go step by step and we start by explaining the reason why we introduced the space $\mathcal Z_\varphi$.
\begin{lemma}\label{lem-xbyz} There exists a constant $C(T)$ independent from $\varphi$ such that for all $N$
$$
\|u\|_{L^2_{\nu_N}, \mX_{T,\varphi}} \leq C(T) \|u_0\|_{L^2_{\rho_N},\mathcal Z_\varphi}.
$$
\end{lemma}

\begin{proof} The ideas of the proof are two fold : the first one is that we can estimate the $\alpha $ Lipschitz continuity by bounding $\partial_t u$ which we know explicitly in terms of $u$ as $u$ is $\nu_N$ almost surely the solution to \eqref{NLSN}, the second one is that $\rho_N$ is invariant under the flow of \eqref{NLSN}.

We recall the definition of the $\|\cdot\|_{\mathcal{X}_{T,\varphi}}$ given in the introduction to be
\begin{equation}\label{norme}
\|u\|_{\mathcal{X}_{T,\varphi}}=
\sup_{t_1,t_2\in [-T,T]}\frac{\|u(t_1)-u(t_2)\|_{\mathcal{X}_\varphi}}{|t_1-t_2|^\alpha}+
\|u\|_{L^\infty([-T,T],\mathcal{X}_\varphi)}
\end{equation}
(we will fix later the weight function $\varphi$).
We observe that, by H\"older inequality,
\begin{eqnarray*}
\|u(t_1)-u(t_2)\|_{\mathcal{X}_\varphi}
&=&
\|(1+\varphi)^{-1}\langle x\rangle^{-6}D_x^\sigma(u(t_1)-u(t_2))\|_{L^2}
\\
&=&
\left\|(1+\varphi)^{-1}\langle x\rangle^{-6}D_x^\sigma\int_{t_1}^{t_2}\partial_\tau u(\tau)d\tau\right\|_{L^2}
\\
&\leq&
|t_1-t_2|^{1/2}\|\partial_t u\|_{L^2([-T,T],\mathcal{X}_\varphi)}
\end{eqnarray*}
and thus for every $\alpha\in(0,\frac12]$, we get 
$$
\sup_{t_1,t_2\in [-T,T]}\frac{\|u(t_1)-u(t_2)\|_{\mathcal{X}_\varphi}}{|t_1-t_2|^\alpha} \leq C(T) (\|\partial_t u\|_{L^2([-T,T],\mX_\varphi)} + \| u\|_{L^2([-T,T],\mX_\varphi)}).
$$
By Sobolev embeddings, we have 
$$
\|u\|_{L^\infty([-T,T],\mathcal{X}_\varphi)} \leq C(T) (\|\partial_t u\|_{L^2([-T,T],\mX_\varphi)} + \| u\|_{L^2([-T,T],\mX_\varphi)}).
$$
We take the $L^2_{\nu_N}$ norm to get
$$
\|u\|_{L^2_{\nu_N},L^\infty([-T,T],\mathcal{X}_\varphi)} \leq C(T) (\|\partial_t u\|_{L^2_{\nu_N},L^2([-T,T],\mX_\varphi)} + \| u\|_{L^2_{\nu_N},L^2([-T,T],\mX_\varphi)}).
$$
We use the definition of $\nu_N$ as the image measure of $\rho_N$ under $\psi_N(t)$ to get
$$
\|u\|_{L^2_{\nu_N},L^\infty([-T,T],\mathcal{X}_\varphi)} \leq C(T) (\|\partial_t \psi_N(t)u_0\|_{L^2_{\rho_N},L^2([-T,T],\mX_\varphi)} + \| \psi_N(t)u_0\|_{L^2_{\rho_N},L^2([-T,T],\mX_\varphi)}).
$$
We can now exchange the norms in probability and in time by Fubini to get
$$
\|u\|_{L^2_{\nu_N},L^\infty([-T,T],\mathcal{X}_\varphi)} \leq C(T) (\|\partial_t \psi_N(t)u_0\|_{L^2([-T,T](L^2_{\rho_N},\mX_\varphi))} + \| \psi_N(t)u_0\|_{L^2([-T,T](L^2_{\rho_N},\mX_\varphi))}).
$$
As $\psi_N$ is the flow of \eqref{NLSN}, we get
$$
\|\partial_t \psi_N(t)u_0\|_{L^2([-T,T](L^2_{\rho_N},\mX_\varphi)}\leq \|\lap \psi_N(t)u_0\|_{L^2([-T,T](L^2_{\rho_N},\mX_\varphi)}+ \|\chi_N |\psi_N(t)u_0|^2\|_{L^2([-T,T](L^2_{\rho_N},\mX_\varphi)}.
$$
We recall that $\mathcal Z_\varphi$ is given by \eqref{spaceZ1}. Thanks to its $L^2$ part, we have 
$$
\|\lap \psi_N(t)u_0\|_{L^2([-T,T](L^2_{\rho_N},\mX_\varphi)} \leq \| \psi_N(t)u_0\|_{L^2([-T,T](L^2_{\rho_N},\mathcal Z_\varphi)}
$$
and 
$$
\| \psi_N(t)u_0\|_{L^2([-T,T](L^2_{\rho_N},\mX_\varphi)} \leq \| \psi_N(t)u_0\|_{L^2([-T,T](L^2_{\rho_N},\mathcal Z_\varphi)}.
$$
And thanks to its $L^6$ part and the fact that $\chi_N \leq 1$, we have 
$$
\|\chi_N |\psi_N(t)u_0|^2\|_{L^2([-T,T](L^2_{\rho_N},\mX_\varphi)} \leq \| \psi_N(t)u_0\|_{L^2([-T,T](L^2_{\rho_N},\mathcal Z_\varphi)}.
$$

Therefore we get 
$$
\|u\|_{L^2_{\nu_N},L^\infty([-T,T],\mathcal{X}_\varphi)} \leq C(T) \| \psi_N(t)u_0\|_{L^2([-T,T](L^2_{\rho_N},\mathcal Z_\varphi)}.
$$
We now use the invariance of $\rho_N$ under $\psi_N(t)$ for the topological $\sigma$-algebra of $\mathcal Z_\varphi$ to get
$$
\|u\|_{L^2_{\nu_N},L^\infty([-T,T],\mathcal{X}_\varphi)} \leq C(T) \| u_0\|_{L^2([-T,T](L^2_{\rho_N},\mathcal Z_\varphi)}
$$
and we take the norm in time to get
$$
\|u\|_{L^2_{\nu_N},L^\infty([-T,T],\mathcal{X}_\varphi)} \leq C(T)\sqrt T \| u_0\|_{L^2_{\rho_N},\mathcal Z_\varphi}
$$
which concludes the proof of the first lemma.\end{proof}

We are left with proving that there exists $\varphi$ such that $\|u\|_{L^2_{\rho_N},\mathcal Z_\varphi}$ is uniformly bounded in $N$.

We divide the problem into two parts by writing $\mathcal Z_\varphi$ as $\mathcal Z_\varphi^2 \cap \mathcal Z_\varphi^6$ with $\mathcal Z_\varphi^2$ the space induced by the norm
$$
\|\an{x}^{-2}(1+\varphi)^{-1} D^{\sigma+2} f\|_{L^2}
$$
and $\mathcal Z_\varphi^6$ the space induced by the norm 
$$
\|\an{x}^{-2}(1+\varphi)^{-1/3}  f\|_{L^6}.
$$

We start with the $L^6$ part as the absence of derivatives makes it easier to deal with.

\begin{lemma}\label{lem-boundL6} Let $\varphi$ be a non negative, even function increasing on $\R^+$ such that $\varphi \geq \varphi_6^{1/2}$ where $\varphi_6$ is the one defined in \eqref{phi1}, then 
$$
\|u\|_{L^2_{\rho_N},\mathcal Z_\varphi^6}
$$
is uniformly bounded in $N$.
\end{lemma}

\begin{proof} As $\rho_N$ is a probability measure we have $\|\cdot \|_{L^2_{\rho_N}} \leq \|\cdot \|_{L^6_{\rho_N}}$, hence we have 
$$
\|u\|_{L^2_{\rho_N},\mathcal Z_\varphi^6}^6 \leq \E_{\rho_N} \Big( \int_{\R }dx \an{x}^{-12}(1+\varphi(x))^{-2}  |u(x)|^6\Big).
$$
We exchange the two integrations to get
$$
\|u\|_{L^2_{\rho_N},\mathcal Z_\varphi^6}^6 \leq \int_{\R }dx \an{x}^{-12}(1+\varphi(x))^{-2}  \E_{\rho_N} \Big( |u(x)|^6\Big).
$$
We use Lemma \ref{lem-theoneforfed} to get
$$
 \E_{\rho_N} \Big( |u(x)|^6\Big) \leq \varphi_6(x)
 $$
which yields
$$
\|u\|_{L^2_{\rho_N},\mathcal Z_\varphi^6}^6 \leq \int_{\R }dx \an{x}^{-12}(1+\varphi(x))^{-2}  \varphi_6(x).
$$
With the choice of $\varphi$, this integral converge and does not depend on $N$.
\end{proof}

We now deal with the $L^2$ part of $\mathcal Z_{\varphi}$.

\begin{lemma}\label{lem-boundL2} Let $s<\frac14$. Let $\xi$ be a smooth positive even function decreasing on $\R_+$ and flat enough in $+\infty$ in the sense that
\begin{itemize}
\item $|D^s \xi(x)|^2 \lesssim \varphi_2(x)^{-1} \an{x}^{-2}$,
\item $| \xi(x)|^2 \lesssim \varphi_{2,s}(x)^{-1} \an{x}^{-3}$,
\item $| \xi(x)|^{1-2s} \lesssim \varphi_2(x)^{-1} \an{x}^{-2}$,
\end{itemize}
where $\varphi_2$ and $\varphi_{2,s}$ are the functions defined in Lemma \ref{lem-theoneforfed}.

Then we get that
$$
\|\xi(x)D^s u\|_{L^2_{\rho_N},L^2(\R)}
$$
is uniformly bounded in $N$.
\end{lemma}

\begin{proof} We have 
$$
\|\xi(x)D^s u\|_{L^2_{\rho_N},L^2(\R)} \leq I + II
$$
with 
$$
I = \|(D^s \xi)(x) u\|_{L^2_{\rho_N},L^2(\R)},\qquad II=\|D^s(\xi(x) u)\|_{L^2_{\rho_N},L^2(\R)}.
$$

Let us start with $I$. We have 
$$
I^2 = \E_{\rho_N} \Big( \int dx |D^s \xi(x)|^2 |u(x)|^2 \Big) 
$$
and we exchange the integrals to get
$$
I^2 = \int dx |D^s \xi(x)|^2 \E_{\rho_N} \Big(  |u(x)|^2 \Big) .
$$
We use the fact that by Lemma \ref{lem-theoneforfed} we have $\E_{\rho_N} \Big(  |u(x)|^2 \Big) \leq \varphi_2(x)$ and our assumptions on $\xi$ to make the integral converge and to get that $I$ is uniformly bounded in $N$.

The quantity $II$ can be written as
$$
II^2 = \E_{\rho_N}\Big( \int_{\R\times \R} dx dy \frac{|\xi(x)u(x) - \xi(y)u(y)|^2}{|x-y|^{1+2s}}\Big).
$$
We use symmetry over $x$ and $y$ to get
$$
II^2 = 2\E_{\rho_N} \Big( \int_{|x|\leq |y|} dx dy \frac{|\xi(x)u(x) - \xi(y)u(y)|^2}{|x-y|^{1+2s}}\Big).
$$

We now divide $II^2$ into two parts as $II^2 \leq A+B$ with 
$$
A = 2\E_{\rho_N} \Big( \int_{|x|\leq |y|} dx dy \frac{|\xi(x) - \xi(y)|^2}{|x-y|^{1+2s}}|u(x)|^2\Big)
$$
and 
$$
B = 2\E_{\rho_N} \Big( \int_{|x|\leq |y|} dx dy \frac{|u(x) - u(y)|^2}{|x-y|^{1+2s}}|\xi(y)|^2\Big).
$$
We exchange the order of integration to get
$$
A = 2\int_{|x|\leq |y|} dx dy \frac{|\xi(x) - \xi(y)|^2}{|x-y|^{1+2s}}\E_{\rho_N} \Big( |u(x)|^2\Big)
$$
and 
$$
B = 2\int_{|x|\leq |y|} dx dy\: \E_{\rho_N} \Big( \frac{|u(x) - u(y)|^2}{|x-y|^{1+2s}}\Big)|\xi(y)|^2.
$$
We use \eqref{phi1}-\eqref{phi2} (notice that the couple $(r,\alpha) = (2,s)$ falls within the assumptions of Lemma \ref{lem-theoneforfed}) to get
$$
A = 2\int_{|x|\leq |y|} dx dy \frac{|\xi(x) - \xi(y)|^2}{|x-y|^{1+2s}}\varphi_2(x)
$$
and as $|x|\leq |y|$,
$$
B = 2\int_{|x|\leq |y|} dx dy  \:\varphi_{2,s}(y)|\xi(y)|^2.
$$
For $B$, we integrate in $x$ to get
$$
B = 4\int_{\R} dy \:|y| \varphi_{2,s}(y)|\xi(y)|^2
$$
and we use the hypothesis on $\xi$ to get this integral converge and is uniformly bounded in $N$.

For $A$, we use the smoothness and flatness of $\xi$ at $\infty$ to get that $\xi'$ is bounded and hence for $|x-y|\leq 1$
$$
\frac{|\xi(x) - \xi(y)|^2}{|x-y|^{1+2s}} \lesssim (|\xi(x)|+|\xi(y)|)^{2-1-2s}
$$
and the fact that $\xi$ is even, decreasing on $\R^+$ and $|x|\leq |y|$ to get
$$
\frac{|\xi(x) - \xi(y)|^2}{|x-y|^{1+2s}} \lesssim |\xi(x)|^{1-2s}.
$$
When $|x-y|\geq 1$ we get
$$
\frac{|\xi(x) - \xi(y)|^2}{|x-y|^{1+2s}} \leq \frac{|\xi(x)|^{2}}{|x-y|^{1+2s}}.
$$
Therefore, we have
$$
A \lesssim  \int_{|x|\leq |y|,|x-y|\leq 1} dx dy |\xi(x)|^{1-2s}\varphi_2(x) + \int_{|x|\leq |y|,|x-y|\geq 1} dx dy \frac{|\xi(x)|^{2}}{|x-y|^{1+2s}}\varphi_2(x).
$$
We drop the restriction $|x|\leq |y|$ and we integrate in $y$. We have that $\int_{|x-y|\leq 1}dy$ is finite and does not depend on $x$ and so is $\int_{|x-y|\geq 1}\frac{dy}{|x-y|^{1+2s}}$, hence
$$
A \lesssim \int_{\R} dx |\xi(x)|^{1-2s}\varphi_2(x) + \int_{\R} dx  |\xi(x)|^{2}\varphi_2(x)
$$
and we use the assumptions on $\xi$ to get that the integral converges and are uniformly bounded in $N$.
\end{proof}

\begin{proof}[Proof of Proposition \ref{propest}. ] By Lemma \ref{lem-xbyz}, we have that it is sufficient to get a $\varphi$ such that
$$
\|u\|_{L^2_{\rho_N},\mathcal Z_\varphi}\leq C
$$
where the constant $C$ does not depend on $N$ to conclude. We take $\varphi \leq \varphi_6^{1/2}$ and such that $(1+\varphi(x)) \an{x}^2 = \xi(x)^{-1}$ with $\xi$ satisfying the assumptions of Lemma \ref{lem-boundL2}. Then, by Lemma \ref{lem-boundL6} and Lemma \ref{lem-boundL2}, we get that 
$
\|u\|_{L^2_{\rho_N},\mathcal Z_\varphi}
$
is uniformly bounded in $N$ which concludes the proof.
\end{proof}

We are now ready to prove the main result of this section.

\begin{proposition}
Let $T>0$ and $\varepsilon>0$. Then the family of measures $(\nu_N)_{N\geq1}$ is tight in $\mathcal{X}^\varepsilon_{T,\varphi}$.
\end{proposition}

\begin{proof}
Let $\delta>0$, and define the set
$$
K_\delta:=\{u\in \mathcal{X}^\varepsilon_{T,\varphi}:\|u\|_{\mathcal{X}_{T,\varphi}}\leq \delta^{-1}\}.
$$
Since the embedding $\mathcal{X}_{T,\varphi}\subset \mathcal{X}^\varepsilon_{T,\varphi}$ is compact (see Corollary \ref{compeb}), we have that the set $K_\delta$ is compact in $\mathcal{X}^\varepsilon_{T,\varphi}$ for any $\varepsilon>0$. Moreover, thanks to \eqref{crucialest} and H\"older inequality, we have
$$
\nu_N(K_\delta^c)\leq \delta\|u\|_{L^1_{\nu_N}\mathcal{X}_{T,\varphi}}\leq \delta C.
$$
Therefore, the family of measures $(\nu_N)_{N\geq1}$ is tight in $\mathcal{X}^\varepsilon_{T,\varphi}$.
\end{proof}

\subsection{Existence of a weak flow for NLS}\label{subsec-weaksol}\label{weakflow}

In this subsection, we use Skorokhod's theorem to prove the existence of a weak flow for NLS.

We apply Skorokhod's theorem to get the following proposition.

\begin{proposition} Up to a subsequence, there exists a probability space $(\Omega, \mathcal F, P)$ and a sequence of random variables $(X_N)_N$ with values in $\mXe_{\varphi,T}$ such that the law of $X_N$ is $\nu_N$  and $X_N$ converges almost surely in $\mXe_{\varphi,T}$ towards a random variable $X$. Besides, for all $t\in \R$, almost surely, we have $X_N(t) = \psi_N(t) Y_N$ with $Y_N = X_N(t=0)$ and the law of $X_N(t)$ is $\rho_N$.
\end{proposition}

\begin{proof} This is a direct application of Skorokhod's theorem, as explained in Subsection \ref{subsec-sko} and of Propositions \ref{prop-linkrvmeas} and \ref{prop-invarrv}. \end{proof}

\begin{proposition} The law of $X(t)$ is $\rho$. \end{proposition}

\begin{proof} We have that $X_N$ almost surely converges towards $X$ in $\mXe_{\varphi,T} = \mathcal C ([-T,T],\mXe_\varphi)$. Hence for all $t\in [-T,T]$, $X_N(t)$ almost surely converges in $\mXe_\varphi$ towards $X(t)$. The almost sure convergence implies the convergence in law. Hence, the law of $X(t)$ in $\mXe_\varphi$ is the limit of $\rho_N$ in $\mXe_\varphi$, that is $\rho$.
\end{proof}

\begin{proposition} The random variable $X$ is almost surely a weak solution to 
$$
i\partial_t u = -\lap u + |u|^2u .
$$
\end{proposition}

\begin{proof}We have that $X_N$ is almost surely a solution in $\mXe_{\varphi,T}$ of
$$
i\partial_t X_N  +\lap X_N - \chi_N |X_N|^2 X_N = 0.
$$
Since almost surely $X_N$ converges in $\mXe_{\varphi,T}$ towards $X$, $i\partial_t X_N$ converges in the sense of distribution towards $i\partial_t X$, and  $\lap X_N$ towards $\lap X$. 

We explain why for almost all $\omega\in \Omega$, there exists a subsequence $\chi_{N_k}|X_{N_k}(\omega)|^2X_{N_k}(\omega)$ which converges towards $|X(\omega)|^2X(\omega)$ by proving that for some $\varphi$, $(1+\varphi)^{-1}\an{x}^{-2}X_N$ converges towards $(1+\varphi)^{-1}\an{x}^{-2}X$ in $L^r(\Omega \times [-T,T] \times \R)$ for all $r\in ]1,\infty[$.

We recall that $\chi_N$ converges towards $1$ in the norm $\|\an{x}^{-1}\cdot \|_{L^{\infty}}$ by construction.

With the same techniques as in Subsection \ref{subsec-tightness}, given that the law of $X_N(t)$ is $\rho_N$ and the law of $X(t)$ is $\rho$, we have that for $s\leq \frac2{r}$, $s\leq 1-\frac3{2r}$,
$$
\|(1+\varphi)^{-1}\an{x}^{-1}  D^s X_N \|_{L^r(\Omega \times [-T,T] \times \R)}
$$
is uniformly bounded in $N$ and that 
$$
\|\an{x}^{-1}(1+\varphi)^{-1}  D^s X\|_{L^r(\Omega \times [-T,T] \times \R)}
$$
is finite for all $s<\frac12$ and $r\in ]1,\infty[$. Let
$$
C = \max(\sup_N \|\an{x}^{-1} (1+\varphi)^{-1} D^s X_N \|_{L^r(\Omega \times [-T,T] \times \R)}, \|\an{x}^{-1} (1+\varphi)^{-1} D^s X\|_{L^r(\Omega \times [-T,T] \times \R)}).
$$

We have 
\begin{multline*}
\|\an{x}^{-2}(1+\varphi)^{-1} (X-X_N)\|_{L^r(\Omega \times [-T,T] \times \R )} \leq \\
\an{R}^{-1} 2C + M^{-s} \an R(1+ \varphi(R))2C + \|\Pi_M 1_{|x|\leq R} (X-X_N)\|_{L^r(\Omega\times [-T,T] \times \R)}.
\end{multline*}
Since $\Pi_M 1_{|x|\leq R}$ projects into a space of finite dimension, we get that $\mXe_\varphi$ and $L^r(\R)$ have equivalent topologies on this space, which yields
$$
\|\Pi_M 1_{|x|\leq R} (X-X_N)\|_{L^r( \R)} \leq C(M,R) \|X-X_N\|_{\mXe_\varphi}
$$
and thus
$$
\|\Pi_M 1_{|x|\leq R} (X-X_N)\|_{L^r(\Omega\times [-T,T] \times \R)} \leq C(M,R) \|\Pi_M 1_{|x|\leq R} (X-X_N)\|_{L^r(\Omega \times [-T,T], \mXe_\varphi)} 
$$
Since $\|X_N\|_{L^r(\Omega \times [-T,T],\mXe_\varphi)}$ is uniformly bounded in $N$ and $\|X-X_N\|_{L^r([-T,T],\mXe_\varphi)} \leq T^{1/r} \|X-X_N\|_{\mXe_{\varphi,T}}$ converges towards $0$, we get by the dominated convergence theorem that 
$$
\|\Pi_M 1_{|x|\leq R} (X-X_N)\|_{L^r(\Omega\times [-T,T] \times \R)} \rightarrow 0
$$
and therefore, so does $\|\an{x}^{-2}(1+\varphi)^{-1} (X-X_N)\|_{L^r(\Omega \times [-T,T] \times \R )}$. With $r=6$, we get that $|X_N|^2X_N$ converges towards $|X|^2 X$ in $\an{x}^{6}(1+\varphi)^{3}L^2(\Omega \times [-T,T] \times \R)$ and hence that $\chi_N |X_N|^2 X_N$ converges towards $|X|^2X$ in $\an{x}^7 (1+\varphi)^{3}L^2(\Omega \times [-T,T] \times \R)$. We deduce from that that for almost all $\omega \in \Omega$, there exists a subsequence $X_{N_k}(\omega)$ such that $\chi_{N_k}|X_{N_k}|^2X_{N_k}$ converges towards $|X|^2X$ in $\an{x}^7L^2([-T,T] \times \R)$ and hence weakly. 

Thus, for almost all $\omega \in \Omega$, there exists a subsequence $X_{N_k}(\omega)$ such that $i\partial_t X_{N_k}(\omega)  +\lap X_{N_k}(\omega) - \chi_{N_k} |X_{N_k}(\omega)|^2 X_{N_k}(\omega)$ goes to $i\partial_t X(\omega) + \lap X(\omega) - |X(\omega)|^2X(\omega)$, which ensures that almost surely and in the sense of distributions
$$
i\partial_t X  +\lap X -  |X|^2 X = 0.
$$
\end{proof}

This concludes the proof of the main theorem.

\begin{definition} Let $\Omega' $ be the set of $\omega \in \Omega$ such that $X(\omega)$ satisfies $i\partial_t u = -\lap u + |u|^2u$. Let $A $ be the image by $X(t=0)$ of $\Omega'$. For all $u_0 \in A$, let
$$
\psi(t)(u_0) = \{ X(t) (\omega) \, |\, \omega \in \Omega'\cap X(0)^{-1}(\{u_0\}) \}.
$$
\end{definition}

This defines a weak flow $\psi(t)$ of $i\partial_t u = -\lap u + |u|^2u $. In particular, we do not have uniqueness of the solution.

\subsection{Invariance of \texorpdfstring{$\rho$}{rho} under the weak flow, further remarks}\label{finalinv}

In this subsection, we interpret $\psi(t)$ and $X(t)$ in terms of measures.

\begin{definition} Let $t\in \R$, we call $\mathcal F_t$ the set of measurable sets $A$ of $\mXe_\varphi$ such that for all $u_0 \in \mXe_\varphi$, if $\psi(t)(u_0) \cap A \neq \phi$ then $\psi(t)(u_0) \subseteq A$.
\end{definition}

\begin{proposition} The set $\mathcal F_t$ is a $\sigma$ algebra included in the topological $\sigma$-algebra of $\mXe_\varphi$. \end{proposition}

\begin{proof} The empty set belongs to $\mathcal F_t$.

Let $A \in \mathcal F_t$ and $A^c$ its complementary. Let $u_0 \in \mXe_\varphi$.

If $\psi(t) (u_0)$ is not included in $A^c$ then, we have that $\psi(t)(u_0) \cap A$ is not empty. Hence, as $A$ belongs to $\mathcal F_t$, we get that $\psi(t)(u_0) $ is included in $A$ and thus $\psi(t)u_0 \cap A^c = \Phi$.

The converse statement is that if $\psi(t)u_0 \cap A^c$ is not empty then $\psi(t)u_0$ is included in $A^c$ and hence $A^c$ belongs to $\mathcal F_t$.

Let $(A_n)_{n\in \N}$ be a sequence of sets of $\mathcal F_t$ and let $A = \bigcup A_n$. Let $u_0 \in \mXe_\varphi$.

If $\psi(t)(u_0) \cap A$ is different from the empty set then there exists $n\in \N$ such that $\psi(t) (u_0)\cap A_n$ is non empty. Hence, $\psi(t)u_0 \subseteq A_n \subseteq A$. Thus $A \in \mathcal F_t$.\end{proof}

\begin{remark} The $\sigma$-algebra $\mathcal F_t$ may be trivial. Indeed, if $\psi(t)(u_0)$ is either equal to the empty set or the full set then $\mathcal F_t$ is trivial. 

Nevertheless, let $A_0 = \{|u_0 \in \mXe_\varphi \,|\, \textrm{Card }( \psi(t)(u_0)) = 1\}$ and assume that there exists $A_t$ measurable such that $A_t$ is included in $\psi(t)(A_0)$ then $\mathcal F_t$ contains at least all the $A_t \cap A$ with $A$ measurable in $\mXe_\varphi$.
\end{remark}

\begin{remark} Let us comment upon the lack of uniqueness of the flow. Assume that the cardinal of $\psi(t)(u_0)$ is strictly more than $1$. Then, there exists $\omega_1$ and $\omega_2$ in $\Omega$ such that $X(t)(\omega_1) \neq X(t)(\omega_2)$ but $X(0)(\omega_1) = X(0)(\omega_2) = u_0$. We recall that for all $\tau \in \R$ and $i=1,2$, $X(\tau)(\omega_i)$ is the limit of $X_N(\tau)(\omega_i)$. Since $X(t)(\omega_1) \neq X(t)(\omega_2)$ we get that there do not exist subsequences such that
$$
\psi_{N_k}(t)(X_{N_k}(0)(\omega_1)) = X_{N_k}(t)(\omega_1) =  X_{N'_k}(t)(\omega_2) = \psi_{N'_k}(t)(X_{N'_k}(0)(\omega_2))
$$
and because of the uniqueness and reversibility of $\psi_{N_k}$ that there do not exist subsequences such that
$$
X_{N_k}(0)(\omega_1) =  X_{N'_k}(0)(\omega_2).
$$
In other words, $X_N(0)(\omega_1)$ has to converge in a different way to $u_0$ from $X_N(\omega_2)$.

Hence, if one could prove that almost surely $X(0)(\omega_1) = X(0)(\omega_2)$ implies for example that $X_N(0)(\omega_1) = X_N(0)(\omega_2)$ for an infinite number of $N$s then one would get uniqueness of the flow. To us, it is not obvious how to prove this or even if this is true, but we expect that if it is possible, one should understand it at the level of the convergence of $\rho_N$ towards $\rho$.
\end{remark}

\begin{definition} Let $A \subseteq \mXe_\varphi$ and $t \in \R$, we call the reverse image of $A$ by $\psi(t)$ the set
$$
\psi(t)^{-1}(A) = \{u_0 \in \mXe_\varphi \, |\, \psi(t)(u_0) \subseteq A\}.
$$
\end{definition}

\begin{proposition} Let $t\in \R$ and $A \in \mathcal F_t$, we have 
$$
X(0)^{-1}(\psi(t)^{-1}(A)) = X(t)^{-1}(A).
$$
\end{proposition}

\begin{proof}Let $\omega \in \Omega$. We have that $\omega$ belongs to $X(0)^{-1}(\psi(t)^{-1}(A))$ if and only if $\psi(t)(X(0)(\omega)) \subset A$. But since $A$ belongs to $\mathcal F_t$ then $\psi(t)(X(0)(\omega)) \subset A$ is equivalent to $X(t)(\omega) \in A$. Indeed, $X(t)(\omega)$ belongs to $\psi(t)(X(0)(\omega))$. Therefore, $\omega \in X(0)^{-1}(\psi(t)^{-1}(A))$ is equivalent to $\omega \in X(t)^{-1}(A)$ which concludes the proof. \end{proof}

\begin{definition} Define $\rho^t$ the transported measure of $\rho$ under $\psi(t)$ on $\mathcal F_t$ as
$$
\rho^t(A) = \rho(\psi(t)^{-1}(A)) := P(X(0)^{-1}(\psi(t)^{-1}(A))).
$$
\end{definition}

\begin{proposition} For all $A \in \mathcal F_t$,
$$
\rho^t(A) = \rho(A).
$$
\end{proposition}

\begin{proof} With the last proposition
$$
\rho^t(A) = P(X(t)^{-1}(A))
$$
and the law of $X(t)$ is $\rho$.
\end{proof}

{\bf Acknowledgements.} The first named author is supported by the FIRB 2012 "Dispersive dynamics, Fourier analysis and variational methods".

\bibliographystyle{amsplain}
\bibliography{fed3} 
\nocite{*}

\end{document}